\documentclass[11pt,twoside]{amsart}
\usepackage{latexsym,amssymb,amsmath}
\usepackage{setspace}

\numberwithin{equation}{section}
\newtheorem{theorem}{Theorem}[section]
\newtheorem{lemma}[theorem]{Lemma}
\newtheorem{proposition}[theorem]{Proposition}
\newtheorem{corollary}[theorem]{Corollary}

\theoremstyle{definition}
\newtheorem{definition}[theorem]{Definition}

\newtheorem{remark}[theorem]{Remark}
\newtheorem{example}[theorem]{Example}

\begin{document}


\title[VANISHING IDEALS OVER ODD CYCLES]{Vanishing Ideals of Affine sets Parameterized  by Odd Cycles}

\author{M. Eduardo Uribe-Paczka}
\address{
Departamento de
Matem\'aticas\\
Escuela Superior de F\'isica y Matem\'aticas\\
Instituto Polit\'ecnico Nacional\\
07300 Mexico City.
} 
\email{muribep1700@alumno.ipn.mx}

\thanks{The first author was supported by CONACyT.
The second author was partially supported by CONACyT
and SNI. The third author was partially supported by SNI}


\author{Eliseo Sarmiento}
\address{
Departamento de
Matem\'aticas\\
Escuela Superior de F\'isica y Matem\'aticas\\
Instituto Polit\'ecnico Nacional\\ 07300 Mexico City.
} 
\email{eliseo@esfm.ipn.mx}

\author{Carlos Renter\'ia M\'arquez}
\address{
Departamento de
Matem\'aticas\\
Escuela Superior de F\'isica y Matem\'aticas\\
Instituto Polit\'ecnico Nacional\\
07300 Mexico City.
} 
\email{renteri@esfm.ipn.mx}


\begin{abstract} 

Let $K$ be a finite field. Let $X^{\ast}$ be a subset of
the affine space $K^{n}$, which is parameterized by odd cycles.  
In this paper we give an explicit Gr\"obner basis for the vanishing ideal, $\mathbf{I}(X^{\ast})$, 
of $X^{\ast}$. We give an explicit formula for the regularity of $\mathbf{I}(X^{\ast})$ and finally 
if $X^{\ast}$ is parameterized by an odd cycle of length $k$, we show that the Hilbert function 
of the vanishing ideal of $X^{\ast}$ can be written as linear combination of Hilbert functions of degenerate torus.        

\end{abstract}

\maketitle

\section{Introduction}

We introduce some basic notions from coding theory. Let $K=\mathbb{F}_q$ be the finite field with 
$q$ elements. We consider the $n$-dimensional vector space $\mathbb{F}_q^{n}$ whose elements are $n$-tuples 
$a=(a_{1},\ldots,a_{n})$ with $a_{i}\in{\mathbb{F}_q}$.

\medskip

A {\it linear code} $C$ over the alphabet $\mathbb{F}_{q}$ is a linear subspace of $\mathbb{F}_q^{n}$. The elements of $C$ are called 
codewords. We call $n$ the {\it length} of the code $C$ and $\dim_{\mathbb{F}_{q}}$ $C$ the {\it dimension} of the code $C$ as an 
$\mathbb{F}_{q}$-vector space. The {\it weight of} an element $a=(a_{1},\ldots,a_{n})\in{\mathbb{F}_q^{n}}$ is defined as
$w(a)=\left| \{ i \mid a_{i} \neq 0  \} \right|$. The {\it minimum distance} $\delta(C)$ of a code $C \neq 0$ is given by: 

\begin{center}
$\delta(C)=\min\{ w(a) \mid 0 \neq a \in{C} \}$.
\end{center}

Recall that the {\it projective space\/} of dimension $s-1$ over $K$, denoted by
$\mathbb{P}^{s-1}_{K}$, is the quotient space
$$(K^{s}\setminus\{0\})/\sim $$
where two points $\alpha$, $\beta$ in $K^{s}\setminus\{0\}$
are equivalent if $\alpha=\lambda{\beta}$ for some $\lambda\in K$. We
denote the
equivalence class of $\alpha$ by $[\alpha]$.
Let $S=K[t_1,\ldots,t_s]=\oplus_{d=0}^\infty S_d$ be a polynomial ring over the field $K$ with the standard grading. 
Let $\mathbb{Y}$ be a subset of $\mathbb{P}_{K}^{s-1}$, where $\mathbb{P}_{K}^{s-1}$ is a projective
space over the field $K$. Fix a degree $d\geq 1$. Let $P_1,\ldots,P_m$ be a set of
representatives for the points of $\mathbb{Y}$ with $m=|\mathbb{Y}|$. 
For each $i$ there is $f_i\in S_d$ such that $f_i(P_i)\neq 0$. Let $P_i=[(a_1,\ldots,a_s)]$, there is at
least one $j$ in $\{1,\ldots,s\}$ such that $a_j\neq 0$. Setting $f_i(t_1,\ldots,t_s)=t_j^d$ one has that $f_i\in S_d$ and
$f_i(P_i)\neq 0$. The {\it evaluation map}, denoted by ${\rm ev}_d$, is defined as:    
 
\begin{equation}
\label{pro-ev-map}
{\rm ev}_d\colon S_d=K[t_1,\ldots,t_s]_d\rightarrow K^{|\mathbb{Y}|},\ \ \ \ \ 
f\mapsto
\left(\frac{f(P_1)}{f_1(P_1)},\ldots,\frac{f(P_m)}{f_m(P_m)}\right).
\end{equation}

The map ${\rm ev}_d$ is well-defined, i.e., it is independent of
the set of representatives that we choose for the points of
$\mathbb{Y}$. The map ${\rm ev}_d$ defines a linear map of $K$-vector spaces. 
The image of $S_d$ under ${\rm ev}_d$, denoted by  $C_\mathbb{Y}(d)$, is called a {\it projective Reed-Muller-type code\/} 
of degree $d$ over $\mathbb{Y}$ \cite{duursma-renteria-tapia,GRT}. It is also called an {\it evaluation code\/} associated to $\mathbb{Y}$.

\medskip

Let $Y$ be a subset of $K^{s}$, and let $\mathbb{Y}$ be the projective closure of
$Y$. As $Y$ is finite, its projective closure is:

$$
\mathbb{Y}=\{[(1,\alpha)]\, \vert\,
\alpha\in Y\}\subset \mathbb{P}_{K}^{s}.
$$

Let $P_1,\ldots,P_m$ be the points of $Y$,
and let $S_{\leq d}$ be the $K$-vector space of all polynomials of $S$ of
degree at most $d$. The {\it evaluation map\/} 

\begin{equation}
\label{affine-ev-map}
{\rm ev}_d^a\colon S_{\leq d}\longrightarrow K^{|Y|},\ \ \ \ \ 
f\mapsto \left(f(P_1),\ldots,f(P_m)\right),
\end{equation}

\noindent defines a linear map of
$K$-vector spaces. The image of ${\rm ev}_d^a$, denoted by $C_{Y}(d)$,
defines a {\it linear code\/}. We call
$C_{Y}(d)$ the {\it affine Reed-Muller-type code\/} of
degree $d$ on $Y$.

\medskip

Let $y^{v_1},\ldots,y^{v_s}$ be a finite set of monomials.
As usual if $v_i=(v_{i1},\ldots,v_{in})\in\mathbb{N}^n$,
then we set
$$
y^{v_i}=y_1^{v_{i1}}\cdots y_n^{v_{in}},\ \ \ \ i=1,\ldots,s,
$$
where $y_1,\ldots,y_n$ are the indeterminates of a ring of
polynomials with coefficients in $K$. Consider the following set
parameterized  by these monomials
$$
X^*:=\{(x_1^{v_{11}}\cdots x_n^{v_{1n}},\ldots,x_1^{v_{s1}}\cdots x_n^{v_{sn}})\in{K^{s}}	\vert\, x_i\in K^*\mbox{ for all }i\}
$$
where $K^*=K\setminus\{0\}$. Following \cite{affine-codes} we call $X^\ast$ an
{\it affine algebraic toric set\/} parameterized  by
$y^{v_1},\ldots,y^{v_s}$. The set $X^\ast$  is a multiplicative group under
componentwise multiplication. Following \cite{affine-codes} we call
$C_{X^\ast}(d)$ a {\it parameterized affine code\/} of
degree $d$. Parameterized affine codes are special types of affine Reed-Muller codes 
in the sense of \cite[p.~37]{tsfasman}. If $s=n=1$ and $v_1=1$, then
$X^*=\mathbb{F}_q^*$ and we obtain the classical Reed-Solomon code of
degree $d$ \cite[p.~42]{stichtenoth}. Some families of evaluation codes have been studied extensively, including several
variations of Reed-Muller codes \cite{delsarte-goethals-macwilliams,duursma-renteria-tapia,
gold-little-schenck,GR,GRT,sorensen}. 

\medskip

The {\it dimension\/} and {\it length\/} of $C_{X^\ast}(d)$
are given by $\dim_K C_{X^\ast}(d)$ and $|X^\ast|$ respectively. The dimension
and length are two of the {\it basic parameters} of a linear code, the third
basic parameter is the {\it minimum
distance\/}. The minimum distance of $C_{X^\ast}(d)$ will be denoted by
$\delta_d$. The basic parameters of $C_{X^\ast}(d)$ are related by the
Singleton bound for the minimum distance
$$
\delta_d\leq |X^\ast|-\dim_KC_{X^\ast}(d)+1.
$$

\medskip

The parameters of evaluation codes over finite fields have been computed in a number
of cases. If $\mathbb{X}^\ast$ is the image of the affine space
$K^{s}$ under the map $K^{s}\rightarrow
\mathbb{P}^{s} $, $x\mapsto [(1,x)]$, the parameters
of $C_{\mathbb{X}^\ast}(d)$ are described in
\cite[Theorem~2.6.2]{delsarte-goethals-macwilliams}.

In this article we focus on linear codes parameterized by the edges of a graph $\mathcal{G}$ (see Definition \ref{asso-graph}) 
which has $m$ components and each component is an odd cycle; all our work is based on the affine space. In \cite{even-cycles}, 
the authors work with codes parameterized by even cycles over the projective space, they find an explicit description 
for a set of generators of the vanishing ideal (see Definition~\ref{def-van-ideal}) associated to an even cycle. In the same article,  
we can also find the length of the code associated to a graph $\mathcal{G}$ with $m$ connected components (see {\rm\cite[Theorem~3.2]{even-cycles}}).
In \cite{GJRE-parmet-cycles}, the authors work with codes parameterized by odd cycles over the projective space 
and they prove that parameterized sets by odd cycles over the projective space are 
projective torus. Therefore, if we work with codes parameterized by odd cycles over the projective space, we get 
parameterized codes over projective torus, and these codes are very well known. There is not any paper that works with 
parameterized codes by odd cycles over the affine space, in contrast with \cite{GJRE-parmet-cycles}, we are going to 
see that if $2 \mid q-1$ then affine sets parameterized by odd cycles are not affine torus.

The contents of this paper are as follows. In
Section~\ref{prelim-codes-graphs} we introduce the preliminaries and explain the connection 
between the codes and graphs. In Section~\ref{vanishing-odd-cycles} we provide an explicit description 
of a set of generators from the vanishing ideal of an affine set parameterized by a graph $\mathcal{G}$ with $m$ connected components, 
where each component is an odd cycle (see Theorem~\ref{second-vanish-ideal-main-theorem}). The set parameterized by the edges 
of a graph $\mathcal{G}$ will be denoted by $X_{\mathcal{G}}^{\ast}$ (see Definition \ref{asso-graph}) .

\medskip 

\begin{definition}
\label{def-van-ideal}

\begin{enumerate}

\item[$(\mathrm{i})$] Let $X \subseteq K^{s}$. We set: 

\begin{center}
$\mathbf{I}(X)=\{f\in{S} \mid f(x_{1},\ldots,x_{s})=0$ $\forall x=(x_{1},\ldots,x_{s})\in{X} \}$.
\end{center}
    
\item[$(\mathrm{ii})$] Let $\mathbb{X} \subseteq \mathbb{P}^{s-1}_{K}$. We set: 

\begin{center}
$\mathbf{I}(\mathbb{X})= \langle\{f\in{S} \mid f$ is homogeneous and $f(x)=0$ $ \forall x\in{\mathbb{X}} \} \rangle $
\end{center}

\end{enumerate}

Clearly $\mathbf{I}(X)$ is an ideal, we will call to $\mathbf{I}(X)$ (resp. $\mathbf{I}(\mathbb{X})$ ) the {\it vanishing ideal } of $X$ 
(resp. $\mathbb{X}$).

\end{definition}

For a set $X^\ast$ parameterized by monomials, the main algebraic fact about $\mathbf{I}(X^{\ast})$ that we
use is a remarkable result of \cite{affine-codes} showing that $\mathbf{I}(X^{\ast})$ is a binomial ideal.  
In Section~\ref{regularity-section} we give an explicit formula for the regularity of $\mathbf{I}(X_{\mathcal{G}}^{\ast})$,
where $\mathcal{G}$ is graph with $m$ connected components and each component is an odd cycle. Finally, in Section~\ref{dim-section},  
if $\mathcal{G}$ is an odd cycle of length $k$, we prove that the Hilbert function of $\mathbf{I}(X_{\mathcal{G}}^{\ast})$ can be 
written as a linear combination of Hilbert functions of degenerate torus.

For all unexplained
terminology and additional information  we refer to
\cite{EisStu} (for the theory of binomial ideals),
\cite{AL,Sta1} (for the theory of polynomial ideals and Hilbert functions).

\section{Preliminaries: Codes Associated to a Graphs}\label{prelim-codes-graphs}

We will use the notation and definitions used in the
introduction. In this section, we introduce the connection between 
graphs and codes and we present the basic theory of Hilbert functions that we
will use later.

\medskip

\begin{theorem}{\rm(Combinatorial Nullstellensatz {\rm\cite[Theorem~1.2]{alon-cn}})}
\label{com-null}
Let $R=K[y_{1},\ldots,y_{n}]$ be a polynomial ring over a field $K$, let $f\in{R}$, and let $a=(a_{1},\ldots,a_{n})\in{\mathbb{N}^{n}}$. Suppose that 
the coefficient of $y^{a}$ in $f$ is non zero and $\deg(f)=a_{1}+\cdots+a_{n}$. If $S_{1},\ldots,S_{n}$ are subsets of $K$, with 
$\left|S_{i}\right| > a_{i}$ for all $i$, then there are $s_{1}\in{S_{1}},\ldots,s_{n}\in{S_{n}}$ such that $f(s_{1},\ldots,s_{n})\neq 0$. 
\end{theorem} 

\medskip

Let $X^\ast \subseteq K^{s-1}$ be an affine algebraic toric set parameterized  by
$y^{v_1},\ldots,y^{v_{s-1}}$. The kernel of the evaluation map ${\rm ev}_d$, defined in
Eq.~(\ref{affine-ev-map}), is $\mathbf{I}(X^{\ast})_{\leq d}$; in other words $\mathbf{I}(X^{\ast})_{\leq d}$ 
is the set of all the polynomials of degree less or equal to $d$ that are in $\mathbf{I}(X^{\ast})$, thus 
there is an isomorphism of $K$-vector spaces: 

\begin{center}
$S_{\leq d} \slash \mathbf{I}(X^{\ast})_{\leq d} \simeq C_{X^\ast}(d)$. 
\end{center}   

The {\it affine Hilbert function} of $\mathbf{I}(X^{\ast})$ is given by: 

\begin{center}
$H_{X^\ast}(d)= \dim_{K} S_{\leq d} \slash \mathbf{I}(X^{\ast})_{\leq d} = \dim_{K} C_{X^\ast}(d)$.
\end{center}

Let $\mathbb{X} \subseteq \mathbb{P}^{s-1}_{K}$ be the projective closure of $X^\ast$ 
and let $C_{\mathbb{X}}(d)$ be a
projective Reed-Muller code of degree $d$. It is shown that the codes $C_{\mathbb{X}}(d)$ and $C_{X^{\ast}}(d)$
have the same basic parameters (see {\rm\cite[Theorem~2.4]{affine-codes}}). The kernel of the
evaluation map ${\rm ev}_d$, defined in
Eq.~(\ref{pro-ev-map}), is precisely $\mathbf{I}(\mathbb{X})_d$ the degree
$d$ piece of $\mathbf{I}(\mathbb{X})$. Hence there is an isomorphism of $K$-vector spaces
$$S_d/\mathbf{I}(\mathbb{X})_d\simeq C_{\mathbb{X}}(d).$$

Two of the basic parameters of $C_{\mathbb{X}}(d)$ can be expressed
using Hilbert functions of standard graded algebras \cite{Sta1}, as we
now  explain. Recall that the
{\it Hilbert function\/} of
$\mathbf{I}(\mathbb{X})$ is given by
$$H_{\mathbb{X}}(d):=\dim_K\,
(S/\mathbf{I}(\mathbb{X}))_d=\dim_K\,
S_d/\mathbf{I}(\mathbb{X})_d=\dim_KC_{\mathbb{X}}(d).$$
The unique polynomial $h_{\mathbb{X}}(t)=\sum_{i=0}^{k-1}c_it^i\in
\mathbb{Z}[t]$ of degree $k-1=\dim(S/\mathbf{I}(\mathbb{X}))-1$ such that $h_{\mathbb{X}}(d)=H_{\mathbb{X}}(d)$ for
$d\gg 0$ is called the {\it Hilbert polynomial\/} of $\mathbf{I}(\mathbb{X})$. The
integer $c_{k-1}(k-1)!$, denoted by ${\rm deg}(S/\mathbf{I}(\mathbb{X}))$, is
called the {\it degree\/} or  {\it multiplicity} of $S/\mathbf{I}(\mathbb{X})$. In our
situation
$h_{\mathbb{X}}(t)$ is a non-zero constant because $S/\mathbf{I}(\mathbb{X})$ has dimension $1$.
Furthermore $h_{\mathbb{X}}(d)=|\mathbb{X}|$ for $d\geq |\mathbb{X}|-1$, see \cite[Lecture
13]{harris}. This means that $|\mathbb{X}|$ equals the {\it degree\/}
of $S/\mathbf{I}(\mathbb{X})$. Thus $H_{\mathbb{X}}(d)$ and ${\rm deg}(S/\mathbf{I}(\mathbb{X}))$ equal the
dimension and the length of $C_{\mathbb{X}}(d)$ respectively. There are algebraic
methods, based on elimination theory and Gr\"obner bases, to compute
the dimension and the length of $C_{\mathbb{X}}(d)$ \cite{algcodes}.

\begin{definition}
The {\it regularity index\/} of $S/\mathbf{I}(\mathbb{X})$, denoted by
${\rm reg}(S/\mathbf{I}(\mathbb{X}))$, is the least integer $p\geq 0$ such that
$h_{\mathbb{X}}(d)=H_{\mathbb{X}}(d)$ for $d\geq p$.
\end{definition}

As $S/\mathbf{I}(\mathbb{X})$ is a $1$-dimensional Cohen-Macaulay graded algebra \cite{geramita-kreuzer-Robbiano}, 
the regularity index of  $S/\mathbf{I}(\mathbb{X})$ is the Castelnuovo-Mumford regularity of  $S/\mathbf{I}(\mathbb{X})$
\cite{Eisenbud-geometry}. By Hilbert-Serre Theorem, the Hilbert series of $S \slash \mathbf{I}(\mathbb{X})$ can be uniquely written as 
$F_{\mathbb{X}}(t)=\frac{f(t)}{1-t}$, where $f$ is a polinomial of degree equal to the regularity of 
$S \slash \mathbf{I}(\mathbb{X})$. From the exact sequence: 

\begin{center}
$0 \rightarrow (S \slash \mathbf{I}(\mathbb{X}))[-1] \stackrel{t_{s}}{\rightarrow} S \slash \mathbf{I}(\mathbb{X}) \rightarrow S \slash (t_{s},\mathbf{I}(\mathbb{X})) \rightarrow 0$,
\end{center}

\noindent we deduce $F(S \slash (t_{s},\mathbf{I}(\mathbb{X})),t)=f(t)$, where $F(S \slash (t_{s},\mathbf{I}(\mathbb{X})),t)$ is the Hilbert series of
$S \slash (t_{s},\mathbf{I}(\mathbb{X}))$.

\medskip

Let $>$ be a monomial order on $S$ and let $\left\langle 0 \right\rangle \neq I \subseteq S$
be an ideal. If $f$ is a non-zero polynomial in $S$, we can write: 

\begin{center}
$f=\lambda_{1}t^{\alpha_{1}}+ \cdots +\lambda_{r}t^{\alpha_{r}}$,
\end{center}

\noindent with $\lambda_{i}\in{K^{\ast}}$ for all $i$ and $t^{\alpha_{1}} > \cdots > t^{\alpha_{r}}$. The {\it leading monomial} $t^{\alpha_{1}}$ of $f$
is denoted by $LM(f)$ and the {\it leading term} $\lambda_{1}LM(f)$ of $f$ is denoted by $LT(f)$. We denote by $LT(I)$ the set of leading terms of 
elements of $I$. The {\it ideal of leading terms} of $I$ is the monomial ideal of $S$ given by:  

\begin{center}
$\left\langle LT(I)\right\rangle$.
\end{center} 

\begin{definition}

Let $\left\langle 0 \right\rangle \neq I \subseteq S$
be and ideal. A monomial $t^{a}$ is called a {\it standard monomial} of $S \slash I$, with respect to $>$, if $t^{a}$ is not the 
leading monomial of any polynomial in $I$, that is, $t^{a} \notin \left\langle LT(I)\right\rangle$. A polynomial $f$ is called 
{\it standard} if $f \neq 0$ and $f$ is $k$-linear combination of standard monomials.

\end{definition}

The set of standard monomials, denoted by $\Delta_{>}(I)$, is called the {\it footprint} of $S \slash I$. The image 
of $\Delta_{>}(I)$, under the canonical map $S \longrightarrow S \slash I$, is a basis of $S \slash I$ as a $K$-vector space. 
In particular If $X^\ast \subseteq K^{s-1}$ is an affine algebraic toric set parameterized  by $y^{v_1},\ldots,y^{v_{s-1}}$, then $H_{X^\ast}(d)$ 
is the number of standard monomials of degree less or equal to $d$.  

\medskip

The affine algebraic toric set parameterized by $y_{1},\ldots,y_{s}$ will be denoted by $T^{s}$. 
We call $T^{s}$ an affine torus; 

\begin{center}
$T^{s}=\{(x_{1},\ldots,x_{s}) \mid x_{i}\in{K^{\ast}}\}$. 
\end{center}

\medskip

It is known that $\mathbf{I}(T^{s})=\left\langle \{t_{i}^{q-1}-1\}_{i=1}^{s}\right\rangle$. 
Let $X^\ast \subseteq K^{s}$ be an affine algebraic toric set parameterized  by
$y^{v_1},\ldots,y^{v_{s}}$. By \cite{affine-codes} we know that $\mathbf{I}(X^{\ast})$
is generated by binomials $t^{a}-t^{b}\in{S}$ where $a,b\in{\mathbb{N}^{s}}$. In addition 
there are a few observations to be made. 

\begin{itemize}

\item Since $X^{\ast} \subseteq T^{s}$, then $\mathbf{I}(T^{s}) \subseteq \mathbf{I}(X^{\ast})$, 
hence $\{t_{i}^{q-1}-1\}_{i=1}^{s} \subseteq \mathbf{I}(X^{\ast})$.

\item Let $f=t^{a}-t^{b}$ be a nonzero binomial of $S$. If $\gcd(t^{a},t^{b}) \neq 1$, then we can factor the greatest common 
divisor $t^{c}$ from both $t^{a}$ and $t^{b}$ to obtain $f=t^{c}(t^{a'}-t^{b'})$, for some $a' , b' \in{\mathbb{N}^{s}}$. Since $t^{c}$ never 
is zero on $T^{s}$, for any $c\in{\mathbb{N}^{s}}$, we deduce that $f\in{\mathbf{I}(X^{\ast})}$ if and only if $t^{a'}-t^{b'}\in{\mathbf{I}(X^{\ast})}$.
Accordingly, when looking for binomials generators of $\mathbf{I}(X^{\ast})$ we may restrict ourselves to those binomials $t^{a}-t^{b}$ such that  
$t^{a}$ and $t^{b}$ have no common divisors.

\item Given $a=(a_{1},\ldots,a_{s})\in{\mathbb{N}^{s}}$, we set $\left|a \right|=a_{1}+\cdots+a_{s}$ and $\mathrm{supp}(a)=\{i \mid a_{i} \neq 0\}$. Then clearly,
$t^{a}$ and $t^{b}$ have no common divisors if and only if $\mathrm{supp}(a) \cap \mathrm{supp}(b)=\emptyset$.

\end{itemize} 

Let $\mathcal{G}$ be a simple graph with vertex set $V_{\mathcal{G}}=\{v_{1},\ldots,v_{n}\}$ and edge set 
$E_{\mathcal{G}}=\{e_{1},\ldots,e_{s}\}$. For and edge $e_{i}=\{v_{j},v_{k}\}$, where $v_{j},v_{k}\in{V_{\mathcal{G}}}$, 
let $\nu_{i}=\mathbf{e}_{j}+\mathbf{e}_{k}\in{\mathbb{N}^{n}}$, where, for $1 \leq j \leq n$, $\mathbf{e}_{j}$ is the $j-$th 
element of the canonical basis of $\mathbb{Q}^{n}$. 

\begin{definition}
\label{asso-graph}
The {\it affine algebraic toric set associated} to $\mathcal{G}$ is the affine toric set parameterized 
by the $n-$tuples $\nu_{1},\ldots,\nu_{s}\in{\mathbb{N}^{n}}$, obtained from the edges of $\mathcal{G}$. If 
$X^{\ast}_{\mathcal{G}}$ is the affine parameterized toric set associated to $\mathcal{G}$ we call its 
associated linear code $C_{X^{\ast}_{\mathcal{G}}}(d)$ the {\it parameterized affine code associated }
to $\mathcal{G}$ and we refer to the vanishing ideal of $X^{\ast}_{\mathcal{G}}$ as the {\it vanishing ideal }
over $\mathcal{G}$.  
    
\end{definition}

If $x=(x_{1},\ldots,x_{n})\in{(K^{\ast})^{n}}$ and $e_{i}=\{v_{j},v_{k}\}$ is an edge of $\mathcal{G}$, 
we set $x^{e_{i}}=x^{\mathbf{e}_{j}+\mathbf{e}_{k}}=x_{j}x_{k}$, so that the structural map 
$\theta :(K^{\ast})^{n} \rightarrow X^{\ast}_{\mathcal{G}}$ is given by $x \rightarrow (x^{e_{1}},\ldots,x^{e_{s}})$. 
If $\mathcal{G}$ is a graph with $m$ connected components, of which $\epsilon$ are non-bipartite, the length of 
$C_{X^{\ast}_{\mathcal{G}}}(d)$ has been determined. 

\begin{theorem}
Let $\mathcal{G}$ be a graph with $m$ connected components, of which $\epsilon$ are non-bipartite. Suppose that 
$\mathcal{G}$ has $n$ vertices. Then:

\begin{center}

$ o(X^{\ast}_{\mathcal{G}})= \left \{ \begin{matrix} (q-1)^{n-m+\epsilon} & \mbox{if } 2 \nmid q-1
\\  & 
\\ \displaystyle\frac{(q-1)^{n-m+\epsilon}}{2^{\epsilon}} & \mbox{if } 2 \mid q-1\end{matrix}\right.$

\end{center} 

\end{theorem}

\begin{proof}
It follows from {\rm\cite[Lemma~3.1]{even-cycles}}. 
\end{proof}

\begin{corollary}
Let $\mathcal{G}$ be a graph with $m$ connected components. Suppose that each component is a $k$-cycle and $k$ is odd. 
Then: 

\begin{center}
$o(X_{\mathcal{G}}^{\ast}) = \left \{ \begin{matrix} (q-1)^{km} & \mbox{if } 2 \nmid q-1
\\  & 
\\ \displaystyle\frac{(q-1)^{km}}{2^{m}} & \mbox{if } 2 \mid q-1\end{matrix}\right. $
\end{center}

\end{corollary}

\section{Vanishing Ideals of Odd Cycles}\label{vanishing-odd-cycles}

In this Section we give an explicit description of the 
vanishing ideal associated to a graph $\mathcal{G}$ with $m$ connected components 
and each component is an odd cycle.

\medskip

We continue to use the notation and definitions used in the
introduction. Let $\mathcal{G}$ be a graph with $m$ connected components, 
suppose that each component is a $k$-cycle with $k=2\gamma+1$. Let
$C_{1}^{k}=x_{1} \cdots x_{k}x_{1}$, $C_{2}^{k}=x_{k+1} \cdots x_{2k}x_{k+1}$, 
$\ldots$, $C_{m-1}^{k}=x_{(m-2)k+1} \cdots x_{(m-1)k}x_{(m-2)k+1}$, 
$C_{m}^{k}=x_{(m-1)k+1} \cdots x_{mk}x_{(m-1)k+1}$ be the $m$ components of $\mathcal{G}$. If $2 \nmid q-1$, 
then $o(X_{\mathcal{G}}^{\ast})=(q-1)^{km}$, therefore $X_{\mathcal{G}}^{\ast}$ is an affine torus.  
Suppose that $2 \mid q-1$. Let:

\begin{center}

$F_{1}=\{ A \subseteq \{1,\ldots,k\} \mid \left|A\right|=\gamma \}$, 

\medskip

$F_{2}=\{ A \subseteq \{k+1,\ldots,2k\} \mid \left|A\right|=\gamma \}$,

$\vdots$ 

$F_{m}=\{ A \subseteq \{(m-1)k+1,\ldots,mk\} \mid \left|A\right|=\gamma \}$. 

\end{center}

Let $A_{i}=\{ t_{\alpha_{1}}^{\frac{q-1}{2}} \cdots t_{\alpha_{\gamma}}^{\frac{q-1}{2}}- t_{w_{1}}^{\frac{q-1}{2}} \cdots t_{w_{\gamma+1}}^{\frac{q-1}{2}} \mid \{\alpha_{1},\ldots,\alpha_{\gamma}\} \in{F_{i}}$ and $ \{w_{1},\ldots,w_{\gamma+1}\}=\{(i-1)k+1,\ldots,ik\}-\{\alpha_{1},\ldots,\alpha_{\gamma}\} \}$. Note that 
$A_{i} \subseteq K[t_{(i-1)k+1},\ldots,t_{ik}]$.  

\begin{lemma}
\label{lemma-contain}
For each $i\in{\{1,\ldots,m\}}$ we have that $A_{i} \subseteq \mathbf{I}(X_{\mathcal{G}}^{\ast})$.
\end{lemma}

\begin{proof}
Let $i\in{\{1,\ldots,m\}}$ and $f\in{A_{i}}$. Let $f=t_{\alpha_{1}}^{\frac{q-1}{2}} \cdots t_{\alpha_{\gamma}}^{\frac{q-1}{2}}- t_{w_{1}}^{\frac{q-1}{2}} \cdots t_{w_{\gamma+1}}^{\frac{q-1}{2}}$ where $\{\alpha_{1},\ldots,\alpha_{\gamma}\} \in{F_{i}}$ and $ \{w_{1},\ldots,w_{\gamma+1}\}=\{(i-1)k+1,\ldots,ik\}-\{\alpha_{1},\ldots,\alpha_{\gamma}\}$. 

\medskip

Let $g=t_{(i-1)k+1}^{\frac{q-1}{2}} \cdots t_{(i-1)k+\gamma}^{\frac{q-1}{2}}- t_{(i-1)k+\gamma+1}^{\frac{q-1}{2}} \cdots t_{ik}^{\frac{q-1}{2}}$. For each $j\in{\{1,\ldots,\gamma\}}$, let: 

\begin{center}
$G_{j}(t_{(i-1)k+1},\ldots,t_{ik})=\left \{ \begin{matrix} 1 & \mbox{if } \alpha_{j}=(i-1)k+j
\\ t_{\alpha_{j}}^{\frac{q-1}{2}}t_{(i-1)k+j}^{\frac{q-1}{2}} & \mbox{if } \alpha_{j} \neq (i-1)k+j \end{matrix}\right.$ 
\end{center} 

Let $h=fG_{1} \cdots G_{\gamma}$. Let $w=(x_{(i-1)k+1}x_{(i-1)k+2},\ldots,x_{ik-1}x_{ik},x_{ik}x_{(i-1)k+1})$ where 
$x_{i}\in{K^{\ast}}$. Then we obtain: 

\begin{center}
$h(w)=f(w)\eta=g(w)$,
\end{center}

\noindent where $\eta \in{K^{\ast}}$. Therefore: 

\begin{center}
$g(w)=(x_{(i-1)k+1}x_{(i-1)k+2})^{\frac{q-1}{2}}(x_{(i-1)k+2}x_{(i-1)k+3})^{\frac{q-1}{2}} \cdots (x_{(i-1)k+\gamma}x_{(i-1)k+\gamma+1})^{\frac{q-1}{2}}-(x_{(i-1)k+\gamma+1}x_{(i-1)k+\gamma+2})^{\frac{q-1}{2}} \cdots (x_{ik}x_{(i-1)k+1})^{\frac{q-1}{2}}$.
\end{center}

Note that $g(w)=x_{(i-1)k+1}^{\frac{q-1}{2}}x_{(i-1)k+\gamma+1}^{\frac{q-1}{2}}-x_{(i-1)k+\gamma+1}^{\frac{q-1}{2}}x_{(i-1)k+1}^{\frac{q-1}{2}}$, then $g(w)=0$. It follows that $f(w)=0$, therefore $f\in{\mathbf{I}(X_{\mathcal{G}}^{\ast})}$.\end{proof}

\medskip

\begin{proposition}
\label{prop-binomial}
Let $f=t^{a}-t^{b}$ be a binomial in $K[t_{1},\ldots,t_{km}]$ with $\mathrm{supp}(a) \cap \mathrm{supp}(b)=\emptyset$. Then $f=g+f'$ where 
$g\in{\mathbf{I}(T^{km})}$ and $f'=t^{a'}-t^{b'}$ is a binomial such that none of its two terms is divisible by any $t_{i}^{q-1}$ for all $i$ and 
$\mathrm{supp}(a') \cap \mathrm{supp}(b')=\emptyset$.
\end{proposition}

\begin{proof}
If $t^{a}$ and $t^{b}$ are not divisible by any  $t_{i}^{q-1}$ for all $i$, we can take $g=0$ and $f=f'$. Then we can write $f$ as: 

\begin{center}
$f=m_{1}t_{\alpha_{1}}^{a_{1}} \cdots t_{\alpha_{n}}^{a_{n}}-m_{2}t_{w_{1}}^{b_{1}} \cdots t_{w_{r}}^{b_{r}}$,
\end{center}

\noindent where $t^{a}=m_{1}t_{\alpha_{1}}^{a_{1}} \cdots t_{\alpha_{n}}^{a_{n}}$ and $t^{b}=m_{2}t_{w_{1}}^{b_{1}} \cdots t_{w_{r}}^{b_{r}}$.
$m_{1}$ and $m_{2}$ are monomials such that none of them is divisible by any $t_{i}^{q-1}$ for all $i$. $m_{1}$ is not divisible by any $t_{\alpha_{i}}$
and $a_{i} \geq q-1$ for all $i$. $m_{2}$ is not divisible by any $t_{w_{j}}$ and $b_{j} \geq q-1$ for all $j$. 

\medskip

By the division algorithm we can write $a_{i}=(q-1)q_{i}+r_{i}$ where $0 \leq r_{i} < q-1$ for all $i=1,\ldots,n$. Note that $q_{i} >0$ for all $i$. Then 
$m_{1}t_{\alpha_{1}}^{a_{1}} \cdots t_{\alpha_{n}}^{a_{n}}=t_{\alpha_{1}}^{r_{1}}m_{1}(t_{\alpha_{1}}^{(q-1)(q_{1}-1)} t_{\alpha_{2}}^{a_{2}}  \cdots t_{\alpha_{n}}^{a_{n}})(t_{\alpha_{1}}^{q-1}-1)$ $ + t_{\alpha_{1}}^{r_{1}}m_{1}(t_{\alpha_{1}}^{(q-1)(q_{1}-1)} \cdots t_{\alpha_{n}}^{a_{n}})$. Let 
$m_{1}'=m_{1}t_{\alpha_{1}}^{a_{1}} \cdots t_{\alpha_{n}}^{a_{n}}$, then: 

\begin{center}
$m_{1}'=\left(\displaystyle\sum_{j=1}^{q_{1}} t_{\alpha_{1}}^{(q-1)(q_{1}-j)} \right)t_{\alpha_{1}}^{r_{1}}m_{1}$ $(t_{\alpha_{2}}^{a_{2}} \cdots t_{\alpha_{n}}^{a_{n}})$ $(t_{\alpha_{1}}^{q-1}-1) + t_{\alpha_{1}}^{r_{1}}m_{1}$ $(t_{\alpha_{2}}^{a_{2}} \cdots t_{\alpha_{n}}^{a_{n}})$. 
\end{center}

If we do the previous analysis with the term $t_{\alpha_{1}}^{r_{1}}m_{1}(t_{\alpha_{2}}^{a_{2}} \cdots t_{\alpha_{n}}^{a_{n}})$ and we continue, 
we get: 

\begin{center}
$m_{1}'=g_{1}+(t_{\alpha_{1}}^{r_{1}}t_{\alpha_{2}}^{r_{2}} \cdots t_{\alpha_{n}}^{r_{n}})m_{1}$, 
\end{center} 

\noindent where $g_{1}\in{\mathbf{I}(T^{km})}$. Let $m_{2}'=m_{2}t_{w_{1}}^{b_{1}} \cdots t_{w_{r}}^{b_{r}}$. By the division algorithm 
we can write $b_{j}=(q-1)q_{j}'+b_{j}'$ where $0 \leq b_{j}' < q-1$ for all $j=1,\ldots,r$. Note that $q_{j}' >0$ for all $j$. If we do with $m_{2}'$
the same procedure that we did with $m_{1}'$, we get: 

\begin{center}
$m_{2}'=g_{2}+ (t_{w_{1}}^{b_{1}'}t_{w_{2}}^{b_{2}'} \cdots t_{w_{r}}^{b_{r}'})m_{2}$, 
\end{center} 

\noindent where $g_{2}\in{\mathbf{I}(T^{km})}$. 
Then $f=g_{1}-g_{2} + (t_{\alpha_{1}}^{r_{1}}t_{\alpha_{2}}^{r_{2}} \cdots t_{\alpha_{n}}^{r_{n}})m_{1}-(t_{w_{1}}^{b_{1}'}t_{w_{2}}^{b_{2}'} \cdots t_{w_{r}}^{b_{r}'})m_{2}$. Let $g=g_{1}-g_{2}$ and $f'=(t_{\alpha_{1}}^{r_{1}}t_{\alpha_{2}}^{r_{2}} \cdots t_{\alpha_{n}}^{r_{n}})m_{1}-(t_{w_{1}}^{b_{1}'}t_{w_{2}}^{b_{2}'} \cdots t_{w_{r}}^{b_{r}'})m_{2}$.\end{proof}

\medskip

\begin{remark}

As a consequence of the last proposition we get $ f\in{\mathbf{I}(X_{\mathcal{G}}^{\ast})}$ if and only if $f' \in{\mathbf{I}(X_{\mathcal{G}}^{\ast})}$. 

\end{remark}

\begin{proposition}
\label{main-prop-one}
Let $i\in{ \{1,\ldots,m \} }$. Let $f=t^{a}-t^{b}\in{K[t_{(i-1)k+1},\ldots,t_{ik}]}$ with $\mathrm{supp}(a) \cap \mathrm{supp}(b)=\emptyset$. Suppose that:

\begin{itemize}

\item All the coordinates of $a$ and $b$ are equal to $\frac{q-1}{2}$. 

\item $a$ and $b$ have no empty support. 

\item $\mathrm{supp}(a) \cup \mathrm{supp}(b)=\{(i-1)k+1,\ldots,ik\}$.

\end{itemize}  

Then we obtain: 

\begin{center}
$f\in{\left\langle \left\{t_{i}^{q-1}-1\right\}_{i=1}^{km}\bigcup \left( \displaystyle\bigcup_{i=1}^{m}{A_i} \right) \right\rangle}$.
\end{center}

\end{proposition}

\begin{proof}
Let $J=\left\langle \left\{t_{i}^{q-1}-1\right\}_{i=1}^{km}\bigcup \left( \displaystyle\bigcup_{i=1}^{m}{A_i} \right) \right\rangle$. We can write 
$f$ as: 

\begin{center}

$f=t_{a_{1}}^{\frac{q-1}{2}} \cdots t_{a_{n}}^{\frac{q-1}{2}}-t_{b_{1}}^{\frac{q-1}{2}} \cdots t_{b_{r}}^{\frac{q-1}{2}}$,

\end{center}

\noindent where $t^{a}=t_{a_{1}}^{\frac{q-1}{2}} \cdots t_{a_{n}}^{\frac{q-1}{2}}$ and $t^{b}=t_{b_{1}}^{\frac{q-1}{2}} \cdots t_{b_{r}}^{\frac{q-1}{2}}$. 
By hypothesis we have that $\mathrm{supp}(a) \cup \mathrm{supp}(b)=\{(i-1)k+1,\ldots,ik\}$, therefore $n+r=k$. Suppose that $r> \gamma$. Let 
$r=\gamma +\beta$. On the other hand we have: 

\begin{center}
$-t_{b_{\gamma+1}}^{\frac{q-1}{2}} \cdots t_{b_{\gamma +\beta}}^{\frac{q-1}{2}}[t_{b_{1}}^{\frac{q-1}{2}} \cdots t_{b_{\gamma}}^{\frac{q-1}{2}}-(t_{b_{\gamma+1}}^{\frac{q-1}{2}} \cdots t_{b_{\gamma +\beta}}^{\frac{q-1}{2}})t^{a}]=-t^{b}+t^{a}(t_{b_{\gamma+1}}^{q-1} \cdots t_{b_{\gamma +\beta}}^{q-1})$, 
\end{center}

\noindent note that $\beta+n=\gamma+1$, therefore $t_{b_{1}}^{\frac{q-1}{2}} \cdots t_{b_{\gamma}}^{\frac{q-1}{2}}-(t_{b_{\gamma+1}}^{\frac{q-1}{2}} \cdots t_{b_{\gamma +\beta}}^{\frac{q-1}{2}})t^{a} \in{J}$. We can write $m'=t^{a}(t_{b_{\gamma+1}}^{q-1} \cdots t_{b_{\gamma +\beta}}^{q-1})$ as 
$m'=t^{a}(t_{b_{\gamma+2}}^{q-1} \cdots t_{b_{\gamma +\beta}}^{q-1})[t_{b_{\gamma+1}}^{q-1}-1]+t^{a}(t_{b_{\gamma+2}}^{q-1} \cdots t_{b_{\gamma +\beta}}^{q-1})$. 
If we do the same with the term $t^{a}(t_{b_{\gamma+2}}^{q-1} \cdots t_{b_{\gamma +\beta}}^{q-1})$ and we continue, we get: 

\begin{center}
$m'=g+t^{a}$,
\end{center}

\noindent where $g\in{\mathbf{I}(T^{km})}$. Then $-t_{b_{\gamma+1}}^{\frac{q-1}{2}} \cdots t_{b_{\gamma +\beta}}^{\frac{q-1}{2}}[t_{b_{1}}^{\frac{q-1}{2}} \cdots t_{b_{\gamma}}^{\frac{q-1}{2}}-(t_{b_{\gamma+1}}^{\frac{q-1}{2}} \cdots t_{b_{\gamma +\beta}}^{\frac{q-1}{2}})t^{a}]=-t^{b}+g+t^{a}$, therefore
$f=-t_{b_{\gamma+1}}^{\frac{q-1}{2}} \cdots t_{b_{\gamma +\beta}}^{\frac{q-1}{2}}[t_{b_{1}}^{\frac{q-1}{2}} \cdots t_{b_{\gamma}}^{\frac{q-1}{2}}-(t_{b_{\gamma+1}}^{\frac{q-1}{2}} \cdots t_{b_{\gamma +\beta}}^{\frac{q-1}{2}})t^{a}]-g$. It follows that $f\in{J}$. If $r\leq \gamma$, it is easy to see that 
$n> \gamma$, then we use the same proof as above.\end{proof}

\begin{lemma}
\label{square-lemma}
Let $f\in{\mathbf{I}(X_{\mathcal{G}}^{\ast})}$, then $f(a_{1}^{2},\ldots,a_{km}^{2})=0$ for all $a=(a_{1},\ldots,a_{km})\in{(\mathbb{F}_{q}^{\ast})^{km}}$. 
\end{lemma}

\begin{proof}
Let $i\in{\{0,1,\ldots,m-1\}}$ and  $\bar{a}=(a_{ik+1}^{2},a_{ik+2}^{2},\ldots,a_{k(i+1)}^{2})$. Let $\alpha_{j}=a_{ik+j}^{2}$ for 
$j=1,\ldots,k$. Let: 

\begin{center}
$x_{ik+2}=\alpha_{1}x_{ik+1}^{-1}$,

\medskip

$x_{ik+3}=\alpha_{2}\alpha_{1}^{-1}x_{ik+1}$, 

\medskip

$x_{ik+4}=\alpha_{3}\alpha_{2}^{-1}\alpha_{1}x_{ik+1}^{-1}$,

\medskip

$\vdots$

\medskip

$x_{ik+k}=\alpha_{k-1}\alpha_{k-2}^{-1} \cdots \alpha_{1}x_{ik+1}$,

\end{center}

\noindent where $x_{ik+1}=a_{ik+k}a_{ik+k-1}^{-1} \cdots a_{ik+1}$. It is easy to see that 
$\bar{a}=(x_{ik+1}x_{ik+2}$ $,x_{ik+2}x_{ik+3},$ $\ldots,x_{ik+k-1}x_{ik+k}$ $,x_{ik+k}x_{ik+1})$. It follows that $(a_{1}^{2},\ldots,a_{km}^{2})\in{X_{\mathcal{G}}^{\ast}}$,
then $f(a_{1}^{2},\ldots,a_{km}^{2})=0$.\end{proof}

\begin{lemma}
Let $i\in{ \{1,\ldots,m \} }$ and $f=t^{a}-t^{b}\in{\mathbf{I}(X_{C_{i}^{k}}^{\ast}) \setminus \{0\} }$ with  $\mathrm{supp}(a) \cap \mathrm{supp}(b)=\emptyset$. Suppose that $\deg(f)_{t_{i}} < q-1$ for all $i$ and $\emptyset \neq \mathrm{supp}(a) \subsetneq \{(i-1)k+1,\ldots,ik\} $. 
Then $\mathrm{supp}(b) \neq \emptyset$. 
\end{lemma}

\begin{proof}
Suppose that $\mathrm{supp}(b) = \emptyset$. Then we get that $f=t^{a}-1$. Let $n=\left|\mathrm{supp}(a)\right|$.
We can write $f$ as:

\begin{center}

$f(t_{\alpha_{1}},\ldots,t_{\alpha_{n}})=t_{\alpha_{1}}^{a_{1}} \cdots t_{\alpha_{n}}^{a_{n}}-1$,

\end{center}

\noindent where $t^{a}=t_{\alpha_{1}}^{a_{1}} \cdots t_{\alpha_{n}}^{a_{n}}$. First we will prove that $a_{i}=\frac{q-1}{2}$ for all $i$. 
Let $\beta \in{\mathbb{F}_{q}^{\ast}}$. By the last lemma we get that $f(\beta^{2},1,\ldots,1)=(\beta^{2})^{a_{1}}-1=0$.
Let $h(t_{\alpha_{1}})=t_{\alpha_{1}}^{a_{1}}-1$, Let $S=\{\alpha^{2} \mid \alpha\in{\mathbb{F}_{q}^{\ast}}\}$. 
It is easy to see that $\left|S\right|=\frac{q-1}{2}$, then $h$ vanishes on $S$. Now we are going to examine 
the following cases: 

\begin{itemize}

\item Suppose that $a_{1} <\frac{q-1}{2}$. By Combinatorial Nullstellensatz we get that $h=0$ and this is a contradiction because $a_{1} >0$.  

\item Suppose that $a_{1}>\frac{q-1}{2}$. Let $a_{1}=\frac{q-1}{2}+i_{1}$ where $0 < i_{1} < \frac{q-1}{2}$. Let 
$h'(t_{\alpha_{1}})=t_{\alpha_{1}}^{i_{1}}-1$, it is clear that $h'=h$ over $S$, then $h'$ vanishes on $S$, 
by Combinatorial Nullstellensatz we get that $h'=0$ and this is a contradiction. 

\end{itemize}

Therefore the only possible case is when  $a_{1}=\frac{q-1}{2}$. In this way we can prove that $a_{i}=\frac{q-1}{2}$ for all $i$. 
Let $\alpha\in{\mathbb{F}_{q}^{\ast} \setminus S}$. Now we are going to examine the following cases: 

\begin{itemize}

\item [1)] Suppose that $(i-1)k+1 \in{\mathrm{supp}(a)}$. If $(i-1)k+2 \notin{\mathrm{supp}(a)}$, then $( \alpha,\alpha ,1,\ldots,1)\in{X_{C_{i}^{k}}^{\ast}}$, 
therefore $f( \alpha,\alpha ,1,\ldots,1)=(\alpha)^{\frac{q-1}{2}}-1=0$ and this is a contradiction because $\alpha \notin{S}$. 
Suppose that $(i-1)k+2 \in{\mathrm{supp}(a)}$, if $(i-1)k+3 \notin{\mathrm{supp}(a)}$, then $(1 ,\alpha, \alpha ,1,\ldots,1)\in{X_{C_{i}^{k}}^{\ast}}$, 
thus $f(1 ,\alpha, \alpha ,1,\ldots,1)=(\alpha)^{\frac{q-1}{2}}-1=0$, and this is a contradiction. If we continue, we get that 
$(i-1)k+1,\ldots,ik-1 \in{\mathrm{supp}(a)}$, then $ik \notin{\mathrm{supp}(a)}$ because $\mathrm{supp(a)} \subsetneq \{(i-1)k+1,\ldots,ik\} $, then 
$(1 ,\ldots,1,\alpha,\alpha )\in{X_{C_{i}^{k}}^{\ast}}$, and 
$f(1 ,\ldots,1,\alpha,\alpha )=(\alpha)^{\frac{q-1}{2}}-1=0$ which is a contradiction.  

\item [2)] Suppose that $(i-1)k+1 \notin{\mathrm{supp}(a)}$. If $(i-1)k+2 \in{\mathrm{supp}(a)}$, then $(\alpha,\alpha,1,\ldots,1)\in{X_{C_{i}^{k}}^{\ast}}$, 
then $f(\alpha,\alpha,1,\ldots,1)=(\alpha)^{\frac{q-1}{2}}-1=0$ and this is a contradiction because $\alpha \notin{S}$. If we continue as in case $1)$, 
we get a contradiction.   

\end{itemize}

It follows that $\mathrm{supp}(b) \neq \emptyset$.\end{proof}

\begin{remark}
\label{main-remark}
Let $i\in{ \{1,\ldots,m \} }$ and $f=t^{a}-t^{b}\in{\mathbf{I}(X_{C_{i}^{k}}^{\ast}) \setminus \{0\} }$ with  $\mathrm{supp}(a) \cap \mathrm{supp}(b)=\emptyset$. 
Suppose that $\deg(f)_{t_{i}} < q-1$ for all $i$. If $\mathrm{supp}(a) = \{(i-1)k+1,\ldots,ik\} $, then $\mathrm{supp}(b) = \emptyset$. 
If we follow the proof of the last lemma, we can prove that all the coordinates of $a$ are equal to $\frac{q-1}{2}$. Therefore: 

\begin{center}
$f=t_{(i-1)k+1}^{\frac{q-1}{2}} \cdots t_{ik}^{\frac{q-1}{2}} -1$. 
\end{center}

Let $h=t_{(i-1)k+1}^{\frac{q-1}{2}} \cdots t_{ik}^{\frac{q-1}{2}}$ and
$h'=t_{(i-1)k+1}^{\frac{q-1}{2}} \cdots t_{(1-1)k+\gamma}^{\frac{q-1}{2}}-t_{(i-1)k+\gamma+1}^{\frac{q-1}{2}} \cdots t_{ik}^{\frac{q-1}{2}}$. $h'$ belongs to 
$J=\left\langle \left\{t_{i}^{q-1}-1\right\}_{i=1}^{km}\bigcup \left( \displaystyle\bigcup_{i=1}^{m}{A_i} \right) \right\rangle$. On the other hand 
we  have: 

\begin{center}
$t_{(i-1)k+\gamma+1}^{\frac{q-1}{2}} \cdots t_{ik}^{\frac{q-1}{2}}h'=h - t_{(i-1)k+\gamma+1}^{q-1} \cdots t_{ik}^{q-1}$. 
\end{center} 

Note that $t_{(i-1)k+\gamma+1}^{q-1} \cdots t_{ik}^{q-1}=t_{(i-1)k+\gamma+2}^{q-1} \cdots t_{ik}^{q-1}(t_{(i-1)k+\gamma+1}^{q-1} -1)+t_{(i-1)k+\gamma+2}^{q-1} \cdots t_{ik}^{q-1}$. If we do the same with the term $t_{(i-1)k+\gamma+2}^{q-1} \cdots t_{ik}^{q-1}$ and we continue, we obtain: 

\begin{center}
$t_{(i-1)k+\gamma+1}^{q-1} \cdots t_{ik}^{q-1}=g+1$,
\end{center}

\noindent where $g\in{\mathbf{I}(T^{km})}$. Therefore $t_{(i-1)k+\gamma+1}^{\frac{q-1}{2}} \cdots t_{ik}^{\frac{q-1}{2}}h'=h-(g+1)=f-g$, then 
$f=t_{(i-1)k+\gamma+1}^{\frac{q-1}{2}} \cdots$ $ t_{ik}^{\frac{q-1}{2}}h'+g$. It follows that $f\in{J}$.\end{remark}

\begin{proposition}
Let $i\in{ \{1,\ldots,m \} }$ and $f=t^{a}-t^{b}\in{\mathbf{I}(X_{C_{i}^{k}}^{\ast}) \setminus \{0\} }$ with  $\mathrm{supp}(a) \cap \mathrm{supp}(b)=\emptyset$. Suppose that $\deg(f)_{t_{i}} < q-1$ for all $i$ and $\emptyset \neq \mathrm{supp}(a) \subsetneq \{(i-1)k+1,\ldots,ik\} $. 
Then  all the coordinates of $a$ and $b$ are equal to $\frac{q-1}{2}$.  
\end{proposition}

\begin{proof}
By hypothesis we have that $\mathrm{supp}(a)\neq \emptyset$, let $n=\left|\mathrm{supp}(a)\right|$. 
As $\mathrm{supp}(a) \subsetneq \{(i-1)k+1,\ldots,ik\} $, then $\mathrm{supp}(b) \neq \emptyset$, let $r=\left|\mathrm{supp}(b)\right|$. 
We can write $f$ as:

\begin{center}

$f(t_{\alpha_{1}},\ldots,t_{\alpha_{n}},t_{w_{1}},\ldots,t_{w_{r}})=t_{\alpha_{1}}^{a_{1}} \cdots t_{\alpha_{n}}^{a_{n}}-t_{w_{1}}^{b_{1}} \cdots t_{w_{r}}^{b_{r}}$,

\end{center}

\noindent where $t^{a}=t_{\alpha_{1}}^{a_{1}} \cdots t_{\alpha_{n}}^{a_{n}}$ and $t^{b}=t_{w_{1}}^{b_{1}} \cdots t_{\alpha_{r}}^{b_{r}}$. Let 
$(\beta,\mu)\in{(\mathbb{F}_{q}^{\ast})^{2}}$. By Lemma~\ref{square-lemma} we get that 
$f(\beta^{2},1,\ldots,1,\mu^{2},1,\ldots,1)=(\beta^{2})^{a_{1}}-(\mu^{2})^{w_{1}}=0$.
Let $h(t_{\alpha_{1}},t_{w_{1}})=t_{\alpha_{1}}^{a_{1}}-t_{w_{1}}^{b_{1}}$, Let $S=\{\alpha^{2} \mid \alpha\in{\mathbb{F}_{q}^{\ast}}\}$. 
It is easy to see that $\left|S\right|=\frac{q-1}{2}$, then $h$ vanishes on $S^{2}$. Now we are going to examine 
the following cases: 

\begin{itemize}

\item Suppose that $a_{1},b_{1} <\frac{q-1}{2}$. By Combinatorial Nullstellensatz we get that $h=0$ and this is a contradiction because $a_{1},b_{1} >0$.  

\item Suppose that $b_{1} <\frac{q-1}{2}$ and $a_{1}>\frac{q-1}{2}$. Let $a_{1}=\frac{q-1}{2}+i_{1}$ where $0 < i_{1} < \frac{q-1}{2}$. Let 
$h'(t_{\alpha_{1}},t_{w_{1}})=t_{\alpha_{1}}^{i_{1}}-t_{w_{1}}^{b_{1}}$, it is clear that $h'=h$ over $S^{2}$, then $h'$ vanishes on $S^{2}$, 
by Combinatorial Nullstellensatz we get that $h'=0$ and this is a contradiction. It is very similar the case $a_{1} <\frac{q-1}{2}$ and $b_{1}>\frac{q-1}{2}$.

\item Suppose that $a_{1},b_{1} >\frac{q-1}{2}$. Let $a_{1}=\frac{q-1}{2}+i_{1}$ and $b_{1}=\frac{q-1}{2}+i_{2}$ where $0 < i_{1},i_{2} < \frac{q-1}{2}$.     
Let $h'(t_{\alpha_{1}},t_{w_{1}})=t_{\alpha_{1}}^{i_{1}}-t_{w_{1}}^{i_{2}}$, it is clear that $h'=h$ over $S^{2}$, then $h'$ vanishes on $S^{2}$, 
by Combinatorial Nullstellensatz we get that $h'=0$ and this is a contradiction.

\item Suppose that $b_{1} <\frac{q-1}{2}$ and $a_{1}=\frac{q-1}{2}$. Then $h(t_{\alpha_{1}},t_{w_{1}})=t_{\alpha_{1}}^{\frac{q-1}{2}}-t_{w_{1}}^{b_{1}}$. 
Let $h'(t_{w_{1}})=1-t_{w_{1}}^{b_{1}}$, it is clear that $h'=h$ over $S^{2}$, then $h'$ vanishes on $S$, 
by Combinatorial Nullstellensatz we get that $h'=0$ and this is a contradiction because $b_{1} >0$. It is very similar the case 
$0 < a_{1} <\frac{q-1}{2}$ and $b_{1}=\frac{q-1}{2}$.

\end{itemize}

As a result, the only possible case is when  $b_{1}=a_{1}=\frac{q-1}{2}$. In this way we can prove that 
$a_{i}=\frac{q-1}{2}$ for all $i$ and $b_{j}=\frac{q-1}{2}$ for all $j$.\end{proof}

\begin{proposition}
\label{supp-proposition}
Let $i\in{ \{1,\ldots,m \} }$ and $f=t^{a}-t^{b}\in{\mathbf{I}(X_{C_{i}^{k}}^{\ast}) \setminus \{0\} }$ with  $\mathrm{supp}(a) \cap \mathrm{supp}(b)=\emptyset$. Suppose that $\deg(f)_{t_{i}} < q-1$ for all $i$. Then $\mathrm{supp}(a) \cup \mathrm{supp}(b)=\{(i-1)k+1,\ldots,ik\}$.
\end{proposition}

\begin{proof}
As $f$ is a nonzero polynomial then $\mathrm{supp}(a) \neq \emptyset $ or $\mathrm{supp}(b) \neq \emptyset$, Suppose that $\mathrm{supp}(a) \neq \emptyset $. 
If $\mathrm{supp}(a)=\{(i-1)k+1,\ldots,ik\}$, then $\mathrm{supp}(b)= \emptyset$, it is clear that $\mathrm{supp}(a) \cup \mathrm{supp}(b)=\{(i-1)k+1,\ldots,ik\}$.
Suppose that $\mathrm{supp}(a) \subsetneq \{(i-1)k+1,\ldots,ik\}$, then we know that $\mathrm{supp}(b) \neq \emptyset$ and all the coordinates of $a$ and $b$
are equal to $\frac{q-1}{2}$. Suppose that $H=\mathrm{supp}(a) \cup \mathrm{supp}(b) \subsetneq \{(i-1)k+1,\ldots,ik\}$. 
Let $S=\{\beta^{2} \mid \beta\in{K^{\ast}}\}$ and $\alpha\in{K^{\ast} \setminus S}$. Now we are going to examine the following 
cases: 

\begin{itemize}

\item [1)] Suppose that $(i-1)k+1 \in{H}$. If $(i-1)k+2 \notin{H}$, then $(\alpha,\alpha,1,\ldots,1)\in{X_{C_{i}^{k}}^{\ast}}$, 
therefore $f(\alpha,\alpha,1,\ldots,1)=0$. Then we show $(\alpha)^{\frac{q-1}{2}}=1$ and this is a contradiction because $\alpha \notin{S}$. 
Suppose that $(i-1)k+2 \in{H}$, if $(i-1)k+3 \notin{H}$, then $(1,\alpha,\alpha,1,\ldots,1)\in{X_{C_{i}^{k}}^{\ast}}$, 
therefore $f(1,\alpha,\alpha,1,\ldots,1)=0$, then we find $(\alpha)^{\frac{q-1}{2}}=1$ and this is a contradiction. 
If we continue we get that $(i-1)k+1,\ldots,ik-1 \in{H}$, then $ik \notin{H}$ because $H \subsetneq \{(i-1)k+1,\ldots,ik\} $, then 
$(1,\ldots,1,\alpha,\alpha)\in{X_{C_{i}^{k}}^{\ast}}$, and 
$f(1,\ldots,1,\alpha,\alpha)=0$ which is a contradiction.  

\item [2)] Suppose that $(i-1)k+1 \notin{H}$. If $(i-1)k+2 \in{H}$, then $(\alpha,\alpha,1,\ldots,1)\in{X_{C_{i}^{k}}^{\ast}}$, 
hence $f(\alpha,\alpha,1,\ldots,1)=0$, then we obtain $(\alpha)^{\frac{q-1}{2}}=1$  and this is a contradiction because $\alpha \notin{S}$.
If we continue as in case $1)$, we get a contradiction.   

\end{itemize}

Then $H=\{(i-1)k+1,\ldots,ik\}$. The proof is very similar if we suppose that $\mathrm{supp}(b) \neq \emptyset$.\end{proof}

\bigskip

\begin{theorem}
\label{vanish-ideal-main-theorem}
Let $\mathcal{G}$ be a graph with $m$ connected components, suppose that each component is a $k$-cycle with $k=2\gamma+1$.
The vanishing ideal of $X_{\mathcal{G}}^{\ast}$ is given by: 

\begin{center}
$\mathbf{I}(X_{\mathcal{G}}^{\ast})=\left\langle \left\{t_{i}^{q-1}-1\right\}_{i=1}^{km}\bigcup \left( \displaystyle\bigcup_{i=1}^{m}{A_i} \right) \right\rangle$.
\end{center}

\end{theorem}

\begin{proof}
Let $J=\left\langle \left\{t_{i}^{q-1}-1\right\}_{i=1}^{km}\bigcup \left( \displaystyle\bigcup_{i=1}^{m}{A_i} \right) \right\rangle$. By Lemma~\ref{lemma-contain} 
we get that $J \subseteq \mathbf{I}(X_{\mathcal{G}}^{\ast})$. Now we will prove the other inclusion. We know that $\mathbf{I}(X_{\mathcal{G}}^{\ast})$ 
is generated by binomials, let $f=t^{a}-t^{b}\in{\mathbf{I}(X_{\mathcal{G}}^{\ast})}$ be a binomial; we can suppose that 
$\mathrm{supp}(a) \cap \mathrm{supp}(b)=\emptyset$. By Proposition~\ref{prop-binomial} we can write $f$ as: 

\begin{center}
$f=g+f'$,
\end{center}

\noindent where $g\in{\mathbf{I}(T^{km})}$ and $f'=t^{a'}-t^{b'}$ is a binomial such that none of its terms is divisible by 
any $t_{i}^{q-1}$ for all $i$ and $\mathrm{supp}(a') \cap \mathrm{supp}(b')=\emptyset$. Therefore $f\in{J}$ if and only if $f'\in{J}$. 
We are going to prove that $f'\in{J}$. Let: 

\begin{center}
$f'=t^{a_{1}'} \cdots t^{a_{m}'}- t^{b_{1}'} \cdots t^{b_{m}'}$,
\end{center}

\noindent where $t^{a'}=t^{a_{1}'} \cdots t^{a_{m}'}$ and $t^{b'}=t^{b_{1}'} \cdots t^{b_{m}'}$. $t^{a_{i}'}$ and $t^{b_{i}'}$ are monomials in 
$K[t_{(i-1)k+1},\ldots,t_{ik}]$ for all $i=1,\ldots,m$. Let $i\in{\{(i-1)k+1,\ldots,ik\}}$. We can write $f'$ as: 

\begin{center}
$f'=(t^{a_{i}'}-t^{b_{i}'})t^{a_{1}'} \cdots t^{a_{i-1}'} t^{a_{i+1}'} \cdots t^{a_{m}'}+t^{b_{i}'}$ $[t^{a_{1}'} \cdots t^{a_{i-1}'} t^{a_{i+1}'} \cdots t^{a_{m}'}-t^{b_{1}'} \cdots t^{b_{i-1}'} t^{b_{i+1}'} \cdots t^{b_{m}'}]$. 
\end{center}

As $f\in{\mathbf{I}(X_{\mathcal{G}}^{\ast})}$, it follows that $f'\in{\mathbf{I}(X_{\mathcal{G}}^{\ast})}$. Let 
$(x_{(i-1)k+1},\ldots,x_{ik})\in{(\mathbb{F}_{q}^{\ast})^{k}}$, then 
$x=(1,\ldots,1,x_{(i-1)k+1}x_{(i-1)k+2},x_{(i-1)k+2}x_{(i-1)k+3},\ldots,x_{ik-1}x_{ik},x_{ik}x_{(i-1)k+1},1,\ldots,1)\in{X_{\mathcal{G}}^{\ast}}$. Let 
$f_{i}=t^{a_{i}'}-t^{b_{i}'}$. As $f'$ vanishes on $x$ we get that $f_{i}$ vanishes on 
$(x_{(i-1)k+1}x_{(i-1)k+2},x_{(i-1)k+2}$ $x_{(i-1)k+3},\ldots,x_{ik-1}x_{ik},x_{ik}x_{(i-1)k+1})$, then $f_{i}\in{\mathbf{I}(X_{C_{i}^{k}}^{\ast})}$.
It is clear that $\mathrm{supp}(a_{i}') \cap \mathrm{supp}(b_{i}')=\emptyset$.

\medskip

Suppose that $f_{i}$ is a nonzero polynomial and $supp(a_{i}') \neq \emptyset$. If $\mathrm{supp}(a_{i}')=\{(i-1)k+1,\ldots,ik\}$, then 
$\mathrm{supp}(b_{i}')=\emptyset$. From Remark~\ref{main-remark} we deduce $f_{i}\in{J}$. If $\mathrm{supp}(a_{i}') \subsetneq \{(i-1)k+1,\ldots,ik\}$, 
we know that $\mathrm{supp}(b_{i}') \neq \emptyset$ and all the coordinates of $a_{i}'$ and $b_{i}'$ are equal to $\frac{q-1}{2}$. 
From Proposition~\ref{supp-proposition} we get that $\mathrm{supp}(a_{i}') \cup \mathrm{supp}(b_{i}')=\{(i-1)k+1,\ldots,ik\}$ and by 
Proposition~\ref{main-prop-one} we find $f_{i}\in{J}$. In any case we obtain $f_{i}\in{J}$.

\medskip

On the other hand we have that: 

\begin{center}
$f'=(t^{a_{1}'}-t^{b_{1}'})t^{a_{2}'} \cdots t^{a_{m}'}+t^{b_{1}'}$ $[t^{a_{2}'} \cdots t^{a_{m}'}-t^{b_{2}'} \cdots t^{b_{m}'}]$. 
\end{center}

If we do the same procedure with the binomial $t^{a_{2}'} \cdots t^{a_{m}'}-t^{b_{2}'} \cdots t^{b_{m}'}$, we prove
$t^{a_{2}'} \cdots t^{a_{m}'}-t^{b_{2}'} \cdots t^{b_{m}'}=(t^{a_{2}'}-t^{b_{2}'})t^{a_{3}'} \cdots t^{a_{m}'}+t^{b_{2}'}$ $[t^{a_{3}'} \cdots t^{a_{m}'}-t^{b_{3}'} \cdots t^{b_{m}'}]$, therefore: 

\begin{center}
$f'=f_{1}t^{a_{2}'} \cdots t^{a_{m}'}+f_{2}t^{b_{1}'}t^{a_{3}'} \cdots$ $ t^{a_{m}'}+t^{b_{1}'}t^{b_{2}'}[t^{a_{3}'} \cdots t^{a_{m}'}-t^{b_{3}'} \cdots t^{b_{m}'}]$.
\end{center}

If we proceed like before with the binomial $t^{a_{3}'} \cdots t^{a_{m}'}-t^{b_{3}'} \cdots t^{b_{m}'}$ and we continue, we get: 

\begin{center}
$f'=\displaystyle\sum_{j=1}^m f_{j}h_{j}$,
\end{center}
 
\noindent where $h_{j}\in{K[t_{1},\ldots,t_{km}]}$ for all $j$. As $f_{j}\in{J}$ for all $j$ we show $f'\in{J}$.\end{proof}

Let $\mathcal{G}_{i}$ be a graph with $m_{i}$ connected components and 
each component is a $k_{i}$-cycle, where $1 \leq i \leq r$. Suppose that $k_{i}=2\gamma_{i}+1$ and 
$C_{i1}^{k_{i}}=x_{1}^{i} \cdots x_{k_{i}}^{i}x_{1}^{i}$, $C_{i2}^{k_{i}}=x_{k_{i}+1}^{i} \cdots x_{2k_{i}}^{i}x_{k_{i}+1}^{i}$, \ldots, 
$C_{im_{i}}^{k_{i}}=x_{(m_{i}-1)k_{i}+1}^{i}\cdots x_{m_{i}k_{i}}^{i}x_{(m_{i}-1)k_{i}+1}^{i}$ are all the components of $\mathcal{G}_{i}$.
Let $\mathcal{G}= \bigcup_{i=1}^{r}{\mathcal{G}_{i}}$. In this case we will work with the polynomial ring 
$K[t_{1}^{1},\ldots,t_{k_{1}m_{1}}^{1},\ldots,t_{1}^{r},\ldots,t_{k_{r}m_{r}}^{r}]$.

\medskip

For each $1 \leq i \leq r$, let:  

\begin{center}

$F_{1}^{i}=\{ A \subseteq \{1,\ldots,k_{i}\} \mid \left|A\right|=\gamma_{i} \}$, 

\medskip

$F_{2}^{i}=\{ A \subseteq \{k_{i}+1,\ldots,2k_{i}\} \mid \left|A\right|=\gamma_{i} \}$,

$\vdots$ 

$F_{m_{i}}^{i}=\{ A \subseteq \{(m_{i}-1)k_{i}+1,\ldots,m_{i}k_{i}\} \mid \left|A\right|=\gamma_{i} \}$. 

\end{center}

Let $A_{j}^{i}=\{ (t_{\alpha_{1}}^{i})^{\frac{q-1}{2}} \cdots (t_{\alpha_{\gamma_{i}}}^{i})^{\frac{q-1}{2}}- (t_{w_{1}}^{i})^{\frac{q-1}{2}} \cdots (t_{w_{\gamma_{i}+1}}^{i})^{\frac{q-1}{2}} \mid \{\alpha_{1},\ldots,\alpha_{\gamma_{i}}\} \in{F_{j}^{i}}$ and $ \{w_{1},\ldots,$ $w_{\gamma_{i}+1}\}=\{(j-1)k_{i}+1,\ldots,jk_{i}\}-\{\alpha_{1},\ldots,\alpha_{\gamma_{i}}\} \}$. Note that $ \bigcup_{j=1}^{m_{i}}{A_{j}^{i}} \subseteq K[t_{1}^{i},\ldots,t_{k_{i}m_{i}}^{i}]$.
The following theorem is a generalization of Theorem~\ref{vanish-ideal-main-theorem}.  

\begin{theorem}
\label{second-vanish-ideal-main-theorem}
Let $\mathcal{G}= \bigcup_{i=1}^{r}{\mathcal{G}_{i}}$. Then the vanishing ideal of $X_{\mathcal{G}}^{\ast}$ is given by:

\begin{center}
$\mathbf{I}(X_{\mathcal{G}}^{\ast})=\left\langle \displaystyle\bigcup_{i=1}^{r}{ \{(t_{j}^{i})^{q-1}-1 \}_{j=1}^{k_{i}m_{i}} } \bigcup \left( \displaystyle\bigcup_{i=1}^{r}{ \displaystyle\bigcup_{j=1}^{m_{i}}{A_{j}^{i} }  }\right) \right\rangle$.
\end{center} 

\end{theorem}

\begin{proof}
Let $J=\left\langle \displaystyle\bigcup_{i=1}^{r}{ \{(t_{j}^{i})^{q-1}-1 \}_{j=1}^{k_{i}m_{i}} } \bigcup \left( \displaystyle\bigcup_{i=1}^{r}{ \displaystyle\bigcup_{j=1}^{m_{i}}{A_{j}^{i} }  }\right) \right\rangle$. The proof is  similar to the proof of Theorem~\ref{vanish-ideal-main-theorem}. 
It follows from Lemma~\ref{lemma-contain} that $J \subseteq \mathbf{I}(X_{\mathcal{G}}^{\ast})$. 
Now we will prove the other inclusion. We know that $\mathbf{I}(X_{\mathcal{G}}^{\ast})$ is generated by binomials, let: 

\begin{center}
$f=t^{a_{1}} \cdots t^{a_{r}}- t^{b_{1}} \cdots t^{b_{r}} \in{\mathbf{I}(X_{\mathcal{G}}^{\ast})}$,
\end{center}

\noindent where $t^{a_{i}}$ and $t^{b_{i}}$ are monomials in $K[t_{1}^{i},\ldots,t_{k_{i}m_{i}}^{i}]$ for all $i$. 
Note that we can write $f$ as: 

\begin{center}
$f=(t^{a_{i}}-t^{b_{i}})t^{a_{1}} \cdots t^{a_{i-1}}t^{a_{i+1}} \cdots t^{a_{r}}+ t^{b_{i}}[t^{a_{1}} \cdots t^{a_{i-1}}t^{a_{i+1}}-t^{b_{1}} \cdots t^{b_{i-1}}t^{b_{i+1}} \cdots t^{b_{r}} ]$. 
\end{center}

Let $f_{i}=t^{a_{i}}-t^{b_{i}}$. As $f\in{\mathbf{I}(X_{\mathcal{G}}^{\ast})}$, it follows that $f_{i}\in{\mathbf{I}(X_{\mathcal{G}_{i}}^{\ast})}$, 
from Theorem~\ref{vanish-ideal-main-theorem} we deduce $f_{i}\in{J}$. It is easy to see that we can write $f$ as: 

\begin{center}
$f=\displaystyle\sum_{i=1}^r f_{i}h_{i}$,
\end{center}  

\noindent where $h_{i}\in{K[t_{1}^{1},\ldots,t_{k_{1}m_{1}}^{1},\ldots,t_{1}^{r},\ldots,t_{k_{r}m_{r}}^{r}]}$ for all $i$, therefore $f\in{J}$.\end{proof}

\section{The Regularity Of The Vanishing Ideal Of Odd Cycles }
\label{regularity-section}

We continue with the notation and definitions used in the
introduction and in the preliminaries. Let $\mathcal{G}_{i}$ be a graph with $m_{i}$ connected components, suppose that 
each component is a $k_{i}$-cycle with $k_{i}=2\gamma_{i}+1$ and $1 \leq i \leq r$. The following proposition is easy to prove. 

\begin{proposition}
\label{grob-basis-prop}
Let $\mathcal{G}= \bigcup_{i=1}^{r}{\mathcal{G}_{i}}$. Using the same notation of Theorem~\ref{second-vanish-ideal-main-theorem}, 
let: 

\begin{center}
$G=\displaystyle\bigcup_{i=1}^{r}{ \{(t_{j}^{i})^{q-1}-1 \}_{j=1}^{k_{i}m_{i}} } \bigcup \left( \displaystyle\bigcup_{i=1}^{r}{ \displaystyle\bigcup_{j=1}^{m_{i}}{A_{j}^{i} }  }\right)$.
\end{center} 

Then $G$ is a Gr\"obner basis for $\mathbf{I}(X_{\mathcal{G}}^{\ast})$ with respect to grlex order.
\end{proposition}

\begin{theorem}
\label{regularity-theorem}
Let $\mathcal{G}$ be a graph with $m$ connected components and each component is a $k$-cycle. Suppose that $k=2\gamma+1$ and let
$S=K[t_{1},\ldots,t_{km}]$, then: 

\begin{center}
$\mathrm{reg}(S \slash \mathbf{I}(X_{\mathcal{G}}^{\ast}))=m(k+\gamma)\left( \displaystyle\frac{q-1}{2} \right)-km$.
\end{center}

\end{theorem}

\begin{proof}
Let $\mathbb{Y}_{\mathcal{G}}^{\ast}$ be the projective closure of $X_{\mathcal{G}}^{\ast}$. We know that for each $d \geq 1$
the codes $C_{X_{\mathcal{G}}^{\ast}}(d)$ and $C_{\mathbb{Y}_{\mathcal{G}}^{\ast}}(d)$ have the same basic parameters
(see {\rm\cite[Theorem~2.4]{affine-codes}}). From Proposition~\ref{grob-basis-prop}
we know that: 

\begin{center}
$G=\left\{t_{i}^{q-1}-1\right\}_{i=1}^{km}\bigcup \left( \displaystyle\bigcup_{i=1}^{m}{A_i} \right)$, 
\end{center}

\noindent is a Gr\"obner basis for $\mathbf{I}(X_{\mathcal{G}}^{\ast})$. On the other hand we know that 
$\mathbf{I}(\mathbb{Y}_{\mathcal{G}}^{\ast})=\mathbf{I}(X_{\mathcal{G}}^{\ast})^{h}$, where $\mathbf{I}(X_{\mathcal{G}}^{\ast})^{h}$
is the homogenization of $\mathbf{I}(X_{\mathcal{G}}^{\ast})$. We are going to homogenize 
with respect to the variable $u$, since $G$ is a Gr\"obner basis for $\mathbf{I}(X_{\mathcal{G}}^{\ast})$ with respect to 
the grlex order, it follows that $G^{h}$ is a Gr\"obner basis for $\mathbf{I}(X_{\mathcal{G}}^{\ast})^{h} \subseteq S[u]$, 
where $S[u]=K[t_{1},\ldots,t_{km},u]$, regarding the order:

\begin{center}
$t^{\delta}u^{a} >_{h} t^{\beta}u^{b} \Leftrightarrow t^{\delta} >_{grlex} t^{\beta}$ or $t^{\delta}=t^{\beta}$ and $a > b$, 
\end{center} 

\noindent where $t^{\delta}$ and $t^{\beta}$ are monomials in $S$. Denote by $R$ the graded ring 
$S[u] \slash \mathbf{I}(X_{\mathcal{G}}^{\ast})^{h}$. Consider $u\in{S[u]}$, 
let $\overline{u}=\mathbf{I}(X_{\mathcal{G}}^{\ast})^{h}+u$. $\overline{u}$ is regular on $R$, then we have the following 
exact sequence of graded $S[u]$-modules: 

\begin{center}
$0 \rightarrow R[-1] \stackrel{\overline{u}}{\rightarrow} R   \rightarrow    R \slash \left\langle \overline{u} \right\rangle \rightarrow  0$,
\end{center}

\noindent where $R[-1]$ is the graded $S[u]$-module obtained by a shift in the graduation, in other words 
$R[-1]_{i}=R_{i-1}$. Since $S[u] \slash \mathbf{I}(X_{\mathcal{G}}^{\ast})^{h}$ is a $1$-dimensional ring, 
the regularity of $S[u] \slash \mathbf{I}(X_{\mathcal{G}}^{\ast})^{h}$ is the least integer $r$ for which 
$H_{\mathbb{Y}_{\mathcal{G}}^{\ast}}(d)$ is equal to some constant for all $d \geq r$. From the last exact sequence
we have 
$H_{\mathbb{Y}_{\mathcal{G}}^{\ast}}(d)-H_{\mathbb{Y}_{\mathcal{G}}^{\ast}}(d-1)=\mathrm{dim}_{K}(R \slash \left\langle \overline{u} \right\rangle  )_{d}$. 
For $d \geq 1$, we define: 

\begin{center}
$h_{d}:=\mathrm{dim}_{K}(R \slash \left\langle \overline{u} \right\rangle  )_{d}=H_{\mathbb{Y}_{\mathcal{G}}^{\ast}}(d)-H_{\mathbb{Y}_{\mathcal{G}}^{\ast}}(d-1)$.
\end{center} 

First we will prove that $\mathrm{reg}(S[u] \slash \mathbf{I}(X_{\mathcal{G}}^{\ast})^{h} ) \leq m(k+\gamma)\left( \displaystyle\frac{q-1}{2} \right)-km$.
Let $\alpha=m(k+\gamma)\left( \displaystyle\frac{q-1}{2} \right)-km$. If we show that $h_{d}=0$ for $d \geq \alpha +1$, then 
$H_{\mathbb{Y}_{\mathcal{G}}^{\ast}}(d-1)=H_{\mathbb{Y}_{\mathcal{G}}^{\ast}}(d)$ for $d-1 \geq \alpha$, and our result follows. 
Let $d \geq \alpha +1$. To show that $h_{d}=0$ for $d \geq \alpha +1$, it is enough to prove that if $g\in{S[u]_{d}}$ is a monomial, 
then: 

\begin{equation}
\label{class-equality}
\left\langle \overline{u} \right\rangle +( \mathbf{I}(X_{\mathcal{G}}^{\ast})^{h} +g)=\left\langle \overline{u} \right\rangle + (\mathbf{I}(X_{\mathcal{G}}^{\ast})^{h}).
\end{equation}

Let $g=t^{a_{1}} \cdots t^{a_{m}} u^{a_{0}}\in{S[u]_{d}}$, where $t^{a_{i}}$ is a monomial in $K[t_{(i-1)k+1},\ldots,t_{ik}]$ for $i=1,\ldots,m$.
If $a_{0} > 0$, it is clear that \ref{class-equality} follows, therefore we will suppose that $a_{0}=0$. For $i\in{\{1,\ldots,m\}}$, let 
$g_{i}=t^{a_{1}} \cdots t^{a_{i-1}} t^{a_{i+1}} \cdots t^{a_{m}}$. 

\medskip

Let $i\in{\{1,\ldots,m\}}$, suppose that there is $w\in{ \{(i-1)k+1,\ldots,ik\} }$ such that $t_{w}^{q-1} \mid t^{a_{i}}$, 
then $t^{a_{i}}=t_{w}^{q-1} t^{c}$, where $t^{c}$ is a monomial in $K[t_{(i-1)k+1},\ldots,t_{ik}]$. Therefore we can write $g$ as: 

\begin{center}
$g=t^{a_{1}} \cdots t^{a_{i-1}} t_{w}^{q-1} t^{c} t^{a_{i+1}} \cdots t^{a_{m}}=g_{i}t^{c}[t_{w}^{q-1}-u^{q-1}]+u^{q-1}g_{i}t^{c}$,
\end{center}

\noindent it is clear that \ref{class-equality} follows, then we will suppose that all the coordinates of each $a_{i}$ are less or equal than 
$q-2$. Now we will suppose  that for each $i\in{\{1,\ldots,m\}}$, the monomial $t^{a_{i}}$ is not divisible by any $LT(f)$, for all 
$f\in{A_{i}^{h}}$. Then we can write $t^{a_{i}}$ as:

\begin{center}
$t^{a_{i}}=(t_{w_{i1}}^{a_{i1}} \cdots t_{w_{i \gamma+1}}^{a_{i \gamma+1}})t_{w_{i \gamma+2}}^{a_{i \gamma+2}} \cdots t_{w_{i k}}^{a_{i k}}$,
\end{center}

\noindent where $0 \leq a_{ij} \leq \frac{q-1}{2}-1$ for all $(i,j)\in{\{1,\ldots,m\} \times \{1,\ldots,\gamma+1\}  }$ and 
$\left| \{w_{i1},\ldots,w_{i k} \} \right|=k$. On the other hand we have that:  

\begin{center}
$\mathrm{deg}(g)= \displaystyle\sum_{i=1}^m  \displaystyle\sum_{j=1}^{\gamma+1} a_{ij} + \displaystyle\sum_{i=1}^m  \displaystyle\sum_{j=\gamma+2}^{k} a_{ij} $,
\end{center}

\noindent it follows that $\mathrm{deg}(g) \leq m(\gamma+1)(\frac{q-1}{2}-1)+m\gamma(q-2)$, therefore: 

\begin{center}
$\alpha +1=m(k+\gamma)\left( \displaystyle\frac{q-1}{2} \right)-km +1 \leq m(\gamma+1)\left( \displaystyle\frac{q-1}{2}-1 \right)+m\gamma(q-2)$, 
\end{center}

\noindent we deduce that: 

\begin{center}
$ \hspace{2.8 cm} mk\left(\displaystyle\frac{q-1}{2}\right) -km +1 \leq m\left(\displaystyle\frac{q-1}{2}\right) -m(\gamma+1)+m\gamma(q-2)$,

$ \hspace{1.4 cm} m\left(\displaystyle\frac{q-1}{2}\right)(k-1) +1 \leq m[k-(\gamma+1)] + m\gamma(q-2)$,

$ \hspace{1 cm} m\gamma(q-1) +1 \leq m\gamma + m\gamma(q-2)$,

$m\gamma(q-1) +1 \leq m\gamma(q-1)$,

$ \hspace{0.7 cm} 1 \leq 0$,

\end{center}

\noindent this is a contradiction, therefore there is $i\in{\{1,\ldots,m\}}$ such that $t^{a_{i}}$ is divisible by $LT(f)$ for some 
$f=t_{\alpha_{1}}^{\frac{q-1}{2}} \cdots t_{\alpha_{\gamma}}^{\frac{q-1}{2}}u^{\frac{q-1}{2}}-t_{w_{1}}^{\frac{q-1}{2}} \cdots t_{w_{\gamma+1}}^{\frac{q-1}{2}}\in{A_{i}^{h}}$. Then $t^{a_{i}}=(t_{w_{1}}^{\frac{q-1}{2}} \cdots t_{w_{\gamma+1}}^{\frac{q-1}{2}})t^{c} $,
where $t^{c}$ is a monomial in $K[t_{(i-1)k+1},\ldots,t_{ik}]$ and all the coordinates of $c$ are between $0$ and $q-2$. We can write $g$ as: 

\begin{center}
$g=-g_{i}t^{c}f+g_{i}t^{c}(t_{\alpha_{1}}^{\frac{q-1}{2}} \cdots t_{\alpha_{\gamma}}^{\frac{q-1}{2}}u^{\frac{q-1}{2}})$,
\end{center}

\noindent then we show \ref{class-equality} follows, thus we have proved that 
$\mathrm{reg}(S[u] \slash \mathbf{I}(X_{\mathcal{G}}^{\ast})^{h} ) \leq \alpha$. 

\medskip

Now we will show that $\alpha \leq \mathrm{reg}(S[u] \slash \mathbf{I}(X_{\mathcal{G}}^{\ast})^{h} )$. If we show that $h_{d} >0$
for $d=\alpha$, then $H_{\mathbb{Y}_{\mathcal{G}}^{\ast}}(d-1)<H_{\mathbb{Y}_{\mathcal{G}}^{\ast}}(d)$, for $d=\alpha$, and our 
result follows. It suffices to find a monomial $M\in{S[u]_{d}}$ such that: 

\begin{equation}
\label{class-not-equal}
\left\langle \overline{u} \right\rangle +( \mathbf{I}(X_{\mathcal{G}}^{\ast})^{h}+M) \neq \left\langle \overline{u} \right\rangle + (\mathbf{I}(X_{\mathcal{G}}^{\ast})^{h}). 
\end{equation}

For $i\in{\{1,\ldots,m\}}$, let $M_{i}=t_{(i-1)k+1}^{q-2} \cdots t_{(i-1)k+\gamma}^{q-2}t_{(i-1)k+\gamma+1}^{ \frac{q-1}{2}-1} \cdots t_{ik}^{ \frac{q-1}{2}-1}$.
Note that $\mathrm{deg}(M_{i})=(k+\gamma)(\frac{q-1}{2})-k$ and $M_{i}$ is not divisible by any $LT(f)$ for all $f\in{G_{i}^{h}}$, where 
$G_{i}=\{t_{j}^{q-1}-1\}_{j=(i-1)k+1}^{ik} \cup A_{i}$. Let: 

\begin{center}
$M= \displaystyle\prod_{i=1}^m M_i$.
\end{center}

It is clear that $M\in{S[u]_{d}}$. Suppose that: 

\begin{center}
$\left\langle \overline{u} \right\rangle +( \mathbf{I}(X_{\mathcal{G}}^{\ast})^{h}+M) = \left\langle \overline{u} \right\rangle + (\mathbf{I}(X_{\mathcal{G}}^{\ast})^{h})$,
\end{center} 

\noindent then we get that $\mathbf{I}(X_{\mathcal{G}}^{\ast})^{h}+M\in{\left\langle \overline{u} \right\rangle}$, therefore
$\mathbf{I}(X_{\mathcal{G}}^{\ast})^{h}+M=\mathbf{I}(X_{\mathcal{G}}^{\ast})^{h}+\tilde{g}u$, where $\tilde{g}\in{S[u]}$, thus we deduce that 
$M-\tilde{g}u\in{\mathbf{I}(X_{\mathcal{G}}^{\ast})^{h}}$, then we can write $M$ as: 

\begin{center}
$M=g' + \tilde{g}u$,
\end{center} 

\noindent where $g' \in{\mathbf{I}(X_{\mathcal{G}}^{\ast})^{h}}$, making $u=0$ in the previous equation we obtain: 

\begin{center} 
$M\in{ \left\langle \{LT(f) \mid f\in{G^{h}}\} \right\rangle  }$,  
\end{center}

\noindent it follows that $M$ is divisible by $LT(f)$ for some $f\in{G^{h}}$, and this is a contradiction, thus \ref{class-not-equal} follows.\end{proof}

\begin{theorem}
Let $\mathcal{G}_{i}$ be a graph with $m_{i}$ connected components, suppose that 
each component is a $k_{i}$-cycle with $k_{i}=2\gamma_{i}+1$ and $1 \leq i \leq l$.  
We will work with the polynomial ring $\tilde{S}=K[t_{1}^{1},\ldots,t_{k_{1}m_{1}}^{1},\ldots,t_{1}^{l},\ldots,t_{k_{l}m_{l}}^{l}]$.
Let $\mathcal{G}= \bigcup_{i=1}^{l}{\mathcal{G}_{i}}$, then:

\begin{center}
$\mathrm{reg}(\tilde{S} \slash \mathbf{I}(X_{\mathcal{G}}^{\ast}))=\displaystyle\sum_{i=1}^l  m_{i}(k_{i}+\gamma_{i})\left( \displaystyle\frac{q-1}{2} \right)-k_{i}m_{i}$.
\end{center}

\end{theorem}

\begin{proof}
We will proceed by induction on $l$. If $l=1$ there is nothing else to do. We are going to suppose that our result follows for $l=r$ and we will prove the formula 
for $l=r+1$. Let $\mathcal{G}'=\bigcup_{i=1}^{r}{\mathcal{G}_{i}}$ and $\mathbb{Y}'$ be the projective closure of $X_{\mathcal{G'}}^{\ast}$. Let 
$S'=K[t_{1}^{1},\ldots,t_{k_{1}m_{1}}^{1},\ldots,t_{1}^{r},\ldots,t_{k_{r}m_{r}}^{r}]$, in the projective space we will work with the ring $S'[u]$. 
We know that the the Hilbert series of $S'[u] \slash \mathbf{I}(\mathbb{Y}')$ is given by: 

\begin{center}
$F_{\mathbb{Y}'}(t)=\displaystyle\frac{f(t)}{1-t}$,
\end{center}  

\noindent where $\mathrm{deg}(f)=\mathrm{reg}(S'[u] \slash \mathbf{I}(\mathbb{Y}'))$ and $F(S'[u] \slash (u,\mathbf{I}(\mathbb{Y}')),t)=f(t)$.
Let $\mathbb{Y}''$ be the projective closure of $X_{\mathcal{G}_{r+1}}^{\ast}$ and 
$S''=K[t_{1}^{r+1},\ldots,t_{k_{r+1}m_{r+1}}^{r+1}]$. We know that the Hilbert series 
of $S''[u] \slash \mathbf{I}(\mathbb{Y}'')$ is given by: 

\begin{center}
$F_{\mathbb{Y}''}(t)=\displaystyle\frac{g(t)}{1-t}$,
\end{center}  

\noindent where $\mathrm{deg}(g)=\mathrm{reg}(S''[u] \slash \mathbf{I}(\mathbb{Y}''))$ and $F(S''[u] \slash (u,\mathbf{I}(\mathbb{Y}'')),t)=g(t)$.
According to {\rm\cite[Proposition~2.2.20, p.42]{monalg}}, we have an isomorphism: 

\begin{center}
$S'[u] \slash (u,\mathbf{I}(\mathbb{Y}')) \otimes_{K} S''[u] \slash (u,\mathbf{I}(\mathbb{Y}'')) \cong S[u] \slash (u,\mathbf{I}(\mathbb{Y}))$,
\end{center}

\noindent where $S=K[t_{1}^{1},\ldots,t_{k_{1}m_{1}}^{1},\ldots,t_{1}^{r+1},\ldots,t_{k_{r+1}m_{r+1}}^{r+1}]$ and $\mathbb{Y}$ is the 
projective closure of $X_{\mathcal{G}}^{\ast}$ with $\mathcal{G}=\bigcup_{i=1}^{r+1}{\mathcal{G}_{i}}$. On the other hand we have that 
$F(S'[u] \slash (u,\mathbf{I}(\mathbb{Y}')) \otimes_{K} S''[u] \slash (u,\mathbf{I}(\mathbb{Y}'')),t)=F(S[u] \slash (u,\mathbf{I}(\mathbb{Y})),t)$, thus: 

\begin{center}
$F(S[u] \slash (u,\mathbf{I}(\mathbb{Y})),t)=F(S'[u] \slash (u,\mathbf{I}(\mathbb{Y}')),t)F(S''[u] \slash (u,\mathbf{I}(\mathbb{Y}'')),t)$,
\end{center} 

\noindent (see {\rm\cite[p.102]{monalg}} ) therefore $F(S[u] \slash (u,\mathbf{I}(\mathbb{Y})),t)=f(f)g(t)$, it follows that
$\mathrm{reg}(S[u] \slash \mathbf{I}(\mathbb{Y}))=\mathrm{reg}(S'[u] \slash \mathbf{I}(\mathbb{Y}'))+\mathrm{reg}(S''[u] \slash \mathbf{I}(\mathbb{Y}''))$.
If we apply inductive hypothesis, our result follows.\end{proof}

\section{Dimension Of Parameterized affine Codes by Odd Cycles}
\label{dim-section}

Let $\mathcal{G}$ be a $k$-cycle and $S=K[t_{1},\ldots,t_{k}]$, suppose that $k=2\gamma+1$. 
Let $F=\{ A \subseteq \{1,\ldots,k\} \mid \left|A \right|=\gamma \}$ and $d \geq 1$. For $H=\{h_{1},\ldots,h_{\gamma}\}\in{F}$, let: 

\begin{center}
$A_{H}(d)=\{ t_{h_{1}}^{a_{h_{1}}} \cdots t_{h_{\gamma}}^{a_{h_{\gamma}}} t_{w_{1}}^{a_{w_{1}}} \cdots t_{w_{\gamma+1}}^{a_{w_{\gamma+1}}} \mid 
\{w_{1},\ldots,w_{\gamma+1}\}=\{1,\ldots,k\}-H,$ $a_{h_{i}}<q-1$ for all $i$, $a_{w_{j}}< \frac{q-1}{2}$ for all $j$ and  
$ \sum_{i=1}^\gamma a_{h_{i}} + \sum_{j=1}^{\gamma+1} a_{w_{j}} \leq d   \}$.
\end{center}

Let $r=\left|F\right|= \binom{k}{\gamma}$ and $F=\{H_{1},\ldots,H_{r}\}$. The Hilbert function of $\mathbf{I}(X_{\mathcal{G}}^{\ast})$ can be 
obtained from the following result. 

\begin{theorem}
\label{dim-theo-one}
Let $\mathcal{G}$ be a $k$-cycle and $S=K[t_{1},\ldots,t_{k}]$, suppose that $k=2\gamma+1$. Using the above notation, we have that: 

\begin{center} 
$H_{X_{\mathcal{G}}^{\ast}}(d)=\left| \displaystyle\bigcup_{i=1}^{r}{A_{H_i}(d) } \right|$. 
\end{center}

\end{theorem}

\begin{proof}
Let $A=\{ t_{w_{1}}^{\frac{q-1}{2}} \cdots t_{w_{\gamma+1}}^{\frac{q-1}{2}}-t_{\alpha_{1}}^{\frac{q-1}{2}} \cdots t_{\alpha_{\gamma}}^{\frac{q-1}{2}} \mid 
\{\alpha_{1},\ldots,\alpha_{\gamma}\}\in{F} $ and $ \{w_{1},\ldots,w_{\gamma+1}\}=\{1,\ldots,k\} \setminus \{\alpha_{1},\ldots,\alpha_{\gamma}\} \}$. We 
know that the vanishing ideal of $X_{\mathcal{G}}^{\ast}$ is given by: 

\begin{center}
$\mathbf{I}(X_{\mathcal{G}}^{\ast})= \left\langle  \{t_{i}^{q-1}-1\}_{i=1}^{k} \cup A \right\rangle $. 
\end{center}

Let $G=\{t_{i}^{q-1}-1\}_{i=1}^{k} \cup A$, by Proposition~\ref{grob-basis-prop}, $G$ is a Gr\"obner basis for $\mathbf{I}(X_{\mathcal{G}}^{\ast})$ 
with respect to grlex order. Let $d \geq 1$, we know that $H_{X_{\mathcal{G}}^{\ast}}(d)$ is the number of standard monomials of degree less or equal 
to $d$. Let: 

\begin{center}
$\Delta_{\mathbf{I}(X_{\mathcal{G}}^{\ast})}(d)=\{ m\in{ \Delta_{>_{grlex}}(\mathbf{I}(X_{\mathcal{G}}^{\ast})) } \mid \deg(m) \leq d \}$,
\end{center}

\noindent we are going to prove that: 

\begin{center}
$\Delta_{\mathbf{I}(X_{\mathcal{G}}^{\ast})}(d)=\displaystyle\bigcup_{i=1}^{r}{A_{H_i}(d) }$. 
\end{center}

Let $m\in{\Delta_{\mathbf{I}(X_{\mathcal{G}}^{\ast})}(d)}$, then $m\in{\Delta_{>_{grlex}}(\mathbf{I}(X_{\mathcal{G}}^{\ast}))}$ and 
$\deg(m) \leq d$. As $m\in{\Delta_{>_{grlex}}(\mathbf{I}(X_{\mathcal{G}}^{\ast}))}$, it follows that 
$m \notin{\left\langle LT(\mathbf{I}(X_{\mathcal{G}}^{\ast})) \right\rangle}$. On the other hand $G$ is a Gr\"obner basis, then $m$ is 
not divisible by any leader monomial of any element of $G$. Therefore, there is $H=\{\alpha_{1},\ldots,\alpha_{\gamma}\}\in{F}$ such that:

\begin{center}
$m=t_{\alpha_{1}}^{a_{\alpha_{1}}} \cdots t_{\alpha_{\gamma}}^{a_{\alpha_{\gamma}}}t_{w_{1}}^{a_{w_{1}}} \cdots t_{w_{\gamma+1}}^{a_{w_{\gamma+1}}}$,
\end{center}  

\noindent where $\{w_{1},\ldots,w_{\gamma+1}\}=\{1,\ldots,k\} \setminus H$, $a_{\alpha_{i}} <q-1$ for all $i$ and $a_{w_{j}} <\frac{q-1}{2}$ for 
all $j$. It follows that $m\in{A_{H}(d)}$, thus: 

\begin{center}
$\Delta_{\mathbf{I}(X_{\mathcal{G}}^{\ast})}(d) \subseteq \displaystyle\bigcup_{i=1}^{r}{A_{H_i}(d) }$.
\end{center} 

The other inclusion is clear.\end{proof}

For $d \geq 1$ we define the following sets: 

\medskip

$A_{\gamma}(d)=\{ t_{1}^{a_{1}} \cdots t_{\gamma}^{a_{\gamma}}t_{\gamma+1}^{a_{\gamma+1}} \cdots t_{k}^{a_{k}} \mid 0 \leq a_{i} < q-1$ for all 
$i=1,\ldots, \gamma$,  $0 \leq a_{j} <\frac{q-1}{2}$ for all $j=\gamma+1,\ldots,k$ and $\sum_{i=1}^k a_{i} \leq d\}$,

\medskip

$A_{\gamma-1}(d)=\{ t_{1}^{a_{1}} \cdots t_{\gamma-1}^{a_{\gamma-1}}t_{\gamma}^{a_{\gamma}} \cdots t_{k}^{a_{k}} \mid 0 \leq a_{i} < q-1$ for all 
$i=1,\ldots, \gamma-1$,  $0 \leq a_{j} <\frac{q-1}{2}$ for all $j=\gamma,\ldots,k$ and $\sum_{i=1}^k a_{i} \leq d\}$,

\begin{center}
$\vdots$
\end{center}

$A_{1}(d)=\{ t_{1}^{a_{1}}t_{2}^{a_{2}} \cdots t_{k}^{a_{k}} \mid 0 \leq a_{1} < q-1$,  $0 \leq a_{j} <\frac{q-1}{2}$ for all 
$j=2,\ldots,k$ and $\sum_{i=1}^k a_{i} \leq d\}$,

\medskip

$A_{0}(d)=\{ t_{1}^{a_{1}}\cdots t_{k}^{a_{k}} \mid 0 \leq a_{i} <\frac{q-1}{2}$ for all 
$i=1,\ldots,k$ and $\sum_{i=1}^k a_{i} \leq d\}$.

\begin{remark}
\label{inter-card}
Let $1 \leq l \leq r$ and $1 \leq i_{1} < \cdots < i_{l} \leq r$. It is easy to see that: 

\begin{center}
$\left| A_{H_{i_{1}}}(d) \cap \cdots \cap A_{H_{i_{l}}}(d)  \right|\in{ \{ \left|A_{\gamma}(d)\right|, \ldots, \left|A_{1}(d)\right|, \left| A_{0}(d) \right| \} }$. 
\end{center}

Therefore if we use Theorem~\ref{dim-theo-one}, we can write $H_{X_{\mathcal{G}}^{\ast}}(d)$ as: 

\begin{center}
$H_{X_{\mathcal{G}}^{\ast}}(d)=\beta_{0} \left| A_{0}(d) \right| +\beta_{1}\left|A_{1}(d)\right|+ \cdots +\beta_{\gamma}\left|A_{\gamma}(d)\right|$,
\end{center}

\noindent where $\beta_{0},\ldots,\beta_{\gamma}$ are integers independent of $d$ and $q$. 

\end{remark}

\begin{proposition}
\label{dege-torus-dim}
Let $0 < i \leq \gamma$ and $S=K[t_{1},\ldots,t_{k}]$. Let $X_{i}^{\ast}=\{ (x_{1},\ldots,x_{i},x_{i+1}^{2},\ldots,x_{k}^{2}) \mid x_{j}\in{K^{\ast}}$ for all 
$j \}$ and $X_{0}^{\ast}=\{(x_{1}^{2},\ldots,x_{k}^{2}) \mid x_{i}\in{K^{\ast}}$ for all $i \}$. Then:

\begin{enumerate}

\item[$(\mathrm{i})$] $H_{X_{i}^{\ast}}(d)=\left| A_{i}(d) \right|$.
    
\item[$(\mathrm{ii})$]  $H_{X_{0}^{\ast}}(d)=\left| A_{0}(d) \right|$.  
     
\end{enumerate}

\end{proposition}

\begin{proof}
We are going to prove $(\mathrm{i})$. First, we are going to prove that: 

\begin{center}
$\mathbf{I}(X_{i}^{\ast})=\left\langle t_{1}^{q-1}-1, \ldots, t_{i}^{q-1}-1,t_{i+1}^{\frac{q-1}{2}}-1,\ldots,t_{k}^{\frac{q-1}{2}}-1 \right\rangle$. 
\end{center}

Let $>_{grlex}$ be the grlex order on $S$ and let $f\in{\mathbf{I}(X_{i}^{\ast})}$. By the division algorithm (see \cite[Theorem~3,p. 64]{CLO})
we can write $f$ as: 

\begin{center}
$f=\displaystyle\sum_{j=1}^i h_{j}(t_{j}^{q-1}-1) +\displaystyle\sum_{j=i+1}^k h_{j}(t_{j}^{\frac{q-1}{2}}-1)+G(t_{1},\ldots,t_{k})$,
\end{center} 

\noindent where $h_{j}\in{S}$ for all $j$ and none term of $G$ is divisible by any 
$t_{1}^{q-1}, \ldots, t_{i}^{q-1},t_{i+1}^{\frac{q-1}{2}},\ldots,t_{k}^{\frac{q-1}{2}}$. By Combinatorial-Nullstellensatz, taking 
$S_{j}=K^{\ast}$ for all $j=1,\ldots,i$ and $S_{j}=\{a^{2} \mid a\in{K^{\ast}} \}$ for all $j=i+1,\ldots,k$, we obtain $G=0$. 
Therefore: 

\begin{center}

$\mathbf{I}(X_{i}^{\ast}) \subseteq \left\langle t_{1}^{q-1}-1, \ldots, t_{i}^{q-1}-1,t_{i+1}^{\frac{q-1}{2}}-1,\ldots,t_{k}^{\frac{q-1}{2}}-1 \right\rangle$.

\end{center}

The other inclusion is clear. Let $G=\{t_{1}^{q-1}-1, \ldots, t_{i}^{q-1}-1,t_{i+1}^{\frac{q-1}{2}}-1,\ldots,t_{k}^{\frac{q-1}{2}}-1\}$, by 
\cite[Theorem~6,p. 85]{CLO}, it follows that $G$ is a Gr\"obner basis for $\mathbf{I}(X_{i}^{\ast})$ with respect to grlex order. For $d \geq 1$ 
we know that $H_{X_{i}^{\ast}}(d)$ is the number of standard monomials of degree less or equal to $d$, thus $(\mathrm{i})$ follows.  
The proof of $(\mathrm{ii})$ is similar.\end{proof}

The sets $X_{0}^{\ast},\ldots,X_{\gamma}^{\ast}$ in Proposition~\ref{dege-torus-dim} are degenerate torus (see \cite[Section~4]{cartesian-codes}).
From Remark~\ref{inter-card} and Proposition~\ref{dege-torus-dim} we get that the Hilbert function of $\mathbf{I}(X_{\mathcal{G}}^{\ast})$ can be 
written as: 

\begin{center}
$H_{X_{\mathcal{G}}^{\ast}}(d)=\beta_{0}H_{X_{0}^{\ast}}(d)+\cdots+\beta_{\gamma}H_{X_{\gamma}^{\ast}}(d)$,
\end{center}  

\noindent where $\beta_{0},\ldots,\beta_{\gamma}$ are integers independent of $d$ and $q$. In other words, we can write $H_{X_{\mathcal{G}}^{\ast}}(d)$
as linear combination of Hilbert functions of degenerate torus. For each $i\in{\{0,\ldots,\gamma\}}$ we can find an explicit formula for 
$H_{X_{i}^{\ast}}(d)$ in \cite[Section~4]{cartesian-codes}; therefore if we want to find an explicit formula for $H_{X_{\mathcal{G}}^{\ast}}(d)$, 
we just need to find the values of $\beta_{0},\ldots,\beta_{\gamma}$.  

\begin{example}
Let $K=\mathbb{F}_{5}$ and $S=K[t_{1},t_{2},t_{3},t_{4},t_{5}]$. Let: 

\begin{center}
$X_{\mathcal{G}}^{\ast}=\{(x_{1}x_{2},x_{2}x_{3},x_{3}x_{4},x_{4}x_{5},x_{5}x_{1}) \mid x_{i}\in{\mathbb{F}_{5}^{\ast}} \}$. 
\end{center} 

Let $X_{0}^{\ast},X_{1}^{\ast},X_{2}^{\ast}$ be the sets defined in Proposition~\ref{dege-torus-dim}. Using Macaulay $2.0$, we obtain: 

\begin{eqnarray*}
&&\left.
\begin{array}{c|c|c|c|c}
 d & H_{X_{2}^{\ast}}(d) & H_{X_{1}^{\ast}}(d) & H_{X_{0}^{\ast}}(d) & H_{X_{\mathcal{G}}^{\ast}}(d) \\
   \hline
 1 &  6                  & 6                   & 6                   & 6                             \\ 
   \hline
 2 &  18                 & 17                  &16                   & 21                            \\ 
   \hline
 9 & 128                 & 64                  & 32                  & 512                           \\ 
\end{array}
\right.
\end{eqnarray*}

Note that $\mathrm{reg}(S \slash \mathbf{I}(X_{\mathcal{G}}^{\ast}))=9$. Then we have the following system of 
equations: 

\begin{center}
$6=6\beta_{2}+6\beta_{1}+6\beta_{0}$

\medskip

$21=18\beta_{2}+17\beta_{1}+16\beta_{0}$

\medskip

$512=128\beta_{2}+64\beta_{1}+32\beta_{0}$

\end{center}

Resolving the previous system, we obtain $\beta_{2}=10$, $\beta_{1}=-15$ and $\beta_{0}=6$. Using Macaulay $2.0$ we can verify
that for $d \geq 1$ it follows the following equality: 

\begin{center}  
$H_{X_{\mathcal{G}}^{\ast}}(d)=10H_{X_{2}^{\ast}}(d)-15H_{X_{1}^{\ast}}(d)+6H_{X_{0}^{\ast}}(d)$. 
\end{center}

If we change the field $K$, then the last equality will remain true.

\end{example}

\medskip

\bibliographystyle{plain}

\end{document}